\documentclass[a4paper, 10pt]{amsart}

\usepackage[latin1]{inputenc}
\usepackage[shortlabels]{enumitem}
\usepackage{nameref}
\usepackage[colorlinks=true]{hyperref}
\usepackage[nospace,noadjust]{cite}
\usepackage[T1]{fontenc}
\usepackage{lmodern}
\usepackage[final, nopatch=footnote]{microtype}
\usepackage{lastpage, setspace}
\usepackage{graphicx}
\usepackage[all]{xy}
\usepackage{amsmath, amsthm, amssymb, bm, tensor, mathrsfs}
\usepackage{verbatim}
\usepackage{color}
\usepackage{tikz}
\usetikzlibrary{decorations.markings, decorations.pathreplacing}
\usetikzlibrary{decorations.pathmorphing, arrows.meta}
\usetikzlibrary{calc}
\usepackage{microtype}

\tikzset{snake it/.style={decorate, decoration=snake}}
\tikzset{snake arrow/.style=
{->,
decorate,
decoration={snake,amplitude=.4mm,segment length=2mm,post length=1mm}},
}
\tikzset{big arrow/.tip={Straight Barb[width=4mm, length=3mm]}}

\usepackage{mathtools}

\DeclarePairedDelimiter\brak{[}{]} 
\DeclarePairedDelimiter\set{\{}{\}}
\DeclarePairedDelimiter\abs{\lvert}{\rvert}

\makeatletter
\newcommand{\displaybump}{\hbox to \@totalleftmargin{\hfil}}
\makeatother

\SelectTips{cm}{}
\newdir{ >}{{}*!/-5pt/\dir{>}}
\newdir{< }{{}*!/+5pt/\dir{<}}

\setlist[enumerate]{leftmargin=1cm}  

\usepackage[disable]{todonotes}

\DeclareMathOperator\Aut{Aut}

\newcommand{\ind}[2]{\brak*{#1 : #2}} 
\DeclareMathOperator{\id}{id}

\DeclareMathOperator{\Sym}{Sym}

\newcommand{\bigast}{\mathop{\scalebox{1.5}{\raisebox{-0.2ex}{$\ast$}}}}

\DeclareMathOperator\bbN{\mathbb{N}}

\DeclareMathOperator\calC{\mathcal{C}}
\DeclareMathOperator\calL{\mathcal{L}}

\theoremstyle{definition}
\newtheorem{theorem}{Theorem}[section]
\newtheorem*{theorem*}{Theorem}
\newtheorem{lemma}[theorem]{Lemma}
\newtheorem*{lemma*}{Lemma}
\newtheorem{corollary}[theorem]{Corollary}
\newtheorem{definition}[theorem]{Definition}
\newtheorem*{definition*}{Definition}
\newtheorem{remark}[theorem]{Remark}
\newtheorem*{remark*}{Remark}
\newtheorem{proposition}[theorem]{Proposition}
\newtheorem*{proposition*}{Proposition}
\newtheorem{example}[theorem]{Example}
\newtheorem*{example*}{Example}
\newtheorem*{sketch of proof}{Sketch of Proof}
\newtheorem*{idea of proof}{Idea of Proof}
\newtheorem{question}[theorem]{Question}

\makeatletter
\newtheorem*{rep@theorem}{\rep@title}
\newcommand{\newreptheorem}[2]{%
\newenvironment{rep#1}[1]{%
 \def\rep@title{#2 \scshape \ref{##1}}
 \begin{rep@theorem}}%
 {\end{rep@theorem}}}
\makeatother

\newreptheorem{theorem}{Theorem}
\newreptheorem{proposition}{Proposition}
\newreptheorem{corollary}{Corollary}
\newreptheorem{lemma}{Lemma}
\newreptheorem{conjecture}{Conjecture}
\newreptheorem{definition}{Definition}

\title{The Scale of (\texorpdfstring{$P$}{P})-closed Groups Acting On Trees}
\author{Marcus Chijoff, Michal Ferov, and Stephan Tornier}
\date{\today}

\begin{document}

\begin{abstract}
Reid--Smith parametrised ($P$)-closed groups acting on trees using graph-based combinatorial structures known as local action diagrams. Properties of the acting (topological) group, such as being locally compact, compactly generated, discrete or simple, are reflected in its local action diagram. In this article, we describe the translations of ($P$)-closed groups and their axes in terms of local action diagrams. As applications, we determine the scale function of ($P$)-closed groups and characterise unimodular as well as uniscalar ($P$)-closed groups. The latter provides one possible answer to a question of Thomas Weigel.\end{abstract}

\maketitle

\section{Introduction}

In the general theory of totally disconnected, locally compact (t.d.l.c.) groups, groups acting on trees play an important role for theoretical and practical reasons.

On the theoretical side, due to the Cayley--Abels graph construction \cite{KM08}, every compactly generated t.d.l.c. group acts vertex-transitively on a connected regular graph $\Gamma$ of finite degree. The action of $G$ on $\Gamma$ lifts to an action on its universal cover, which is a regular tree. Note that when $G$ is discrete, compact generation amounts to finite generation and $\Gamma$ is a standard Cayley graph of $G$.

On the practical side, groups acting on trees constitute a particularly accessible class of t.d.l.c. groups which makes them a popular testing ground for conjectures about general t.d.l.c. groups and the construction of (counter)examples.

Many useful classes of groups acting on trees have been defined and parametrised. This includes Burger--Mozes' \cite{BM00} universal groups $\mathrm{U}(F)$, classifying locally transitive ($P$)-closed groups that contain an edge inversion, Smith's \cite{Smi17} universal groups $\mathrm{U}(F_{1},F_{2})$, classifying locally transitive ($P$)-closed groups preserving the bipartition of a biregular tree, Tornier's \cite{Tor23} generalised universal groups $\mathrm{U}_{k}(F)$, classifying locally transitive ($P_{k}$)-closed ($k\in\mathbb{N}$) groups that contain an involutive edge inversion, as well as Radu's \cite{Rad17} groups, classifying locally alternating or symmetric groups that act boundary-$2$-transitively on sufficiently thick trees. 

In striking recent work, Reid--Smith \cite{RS26} achieved a parametrisation of general ($P$)-closed groups acting on trees in terms of \emph{local action diagrams}, see Section~\ref{sec:local_action_diagrams}.

\begin{repdefinition}{def:local_action_diagram}[{\cite[Definition 3.1]{RS26}}] A \textbf{local action diagram} $\Delta$ is a triple $(\Gamma, (X_{a})_{a \in A\Gamma}, (G(v))_{v \in V\Gamma})$ consisting of
\begin{enumerate}[(i)]
	\item a connected graph $\Gamma=(V\Gamma,A\Gamma,o,t,r)$,
	\item pairwise disjoint, non-empty sets $X_{a}$ ($a\in A\Gamma$), and
	\item closed subgroups $G(v)\le\Sym(X_{v})$ ($v\in V\Gamma$), where $\smash{X_{v}:=\bigsqcup_{a\in o^{-1}(v)}X_{a}}$, such that the sets $X_{a}$ ($\smash{a\in o^{-1}(v)}$) are precisely the orbits of $G(v)$.
\end{enumerate}
Each $X_{a}$ is a \textbf{colour set}, its elements are \textbf{colours}, and each $G(v)$ is a \textbf{local action}.
\end{repdefinition}

Reid--Smith introduce isomorphisms for local action diagrams and establish a one-to-one correspondence of their isomorphism classes with isomorphism classes of actions $(T,G)$, where $T$ is a tree and $G\le\Aut(T)$ is ($P$)-closed. Given any group $G\le\Aut(T)$, a local action diagram arises by labelling $\Gamma:=G\backslash T$ with data coming from the group action. Conversely, for every local action diagram $\Delta=(\Gamma,(X_{a}),(G(v)))$ there is an associated tree $\mathbf{T}$, termed $\Delta$-tree, and a ($P$)-closed group $\mathbf{U}(\mathbf{T},(G(v))_v)\le\Aut(\mathbf{T})$, or $\mathbf{U}_{\mathbf{T}}(\Delta)$ in short, such that $\mathbf{U}_{\mathbf{T}}(\Delta)\backslash\mathbf{T}\cong\Gamma$.

Subsequently, they describe various properties of the group $G$ in terms of its local action diagram, including local compactness, compact generation and simplicity.

\vspace{0.2cm}
In this article, we describe the translations of a ($P$)-closed group $G$ and their translation axes in terms of \emph{translatable circuits} of the local action diagram $\Delta$ of $G$.

\vspace{0.2cm}
\noindent
\textbf{Definition \ref{def:multi_coloured_circuit}} and \textbf{\ref{def:translatable_circuit}.}\hspace{0.1cm}
Let $\Delta = (\Gamma, (X_{a}), (G(v)))$ be a local action diagram and $l \in \mathbb{N}$. A~\textbf{translatable circuit of length $l$} of $\Delta$ is a tuple $C=(a_{i}, S_{i})_{i=0}^{l-1}$ consisting of arcs $a_{i}\in A\Gamma$ satisfying $t(a_{i}) = o(a_{i+1})$ for all $i \in \set*{0,\dots, l\!-\!1}$ cyclically, and non-diagonal orbits $S_{i}\subseteq X_{\overline{a_{i-1}}}\times X_{a_{i}}$ of the diagonal action of $G(o(a_{i}))$ on $X_{\overline{a_{i-1}}} \times X_{a_{i}}$.


\vspace{0.2cm}
We introduce two equivalence relations, $\sim$ and $\sim_{s}$, on the set $C_{\Delta}$ of translatable circuits of $\Delta$. Their equivalence classes correspond to $G$-orbits of translation axes and $G$-conjugacy classes of translations of $G$ respectively.

\begin{reptheorem}{thm:axes_circuits_correspondence}
	Let $\Delta=(\Gamma,(X_{a}),(G(v)))$ be a local action diagram, $\mathbf T =(T, \pi, \mathcal{L})$ be a $\Delta$-tree and $G:=\mathbf{U}(\mathbf{T},(G(v)))$. Then there are mutually inverse bijections
    \begin{displaymath}
        \xymatrix{
        G\backslash\mathrm{Axes}(G) \ar@^{->}@<0.5ex>[r] & C_{\Delta}/\!\sim \ar@^{->}@<0.5ex>[l]
        }
        \hspace{0.3cm}\text{and}\hspace{0.3cm}
        \xymatrix{
        G\backslash\mathrm{Hyp}(G) \ar@^{->}@<0.5ex>[r] & C_{\Delta}/\!\sim_{s}. \ar@^{->}@<0.5ex>[l]
        }
    \end{displaymath}
\end{reptheorem}

We show that the scale of a translation (equivalently, of its conjugacy class) depends only on the $\sim_{s}$-equivalence class of the associated translatable circuit $C$. Corollary~\ref{cor:scale_translatable_circuit} offers an explicit formula for this scale value, say $s(C)$, in terms of the translatable circuit's data. As a consequence, we obtain the following description of the set of all scale values of $G$ in terms of translatable circuits of minimal length.

\begin{reptheorem}{thm:scale_values}
    Let $\Delta$ be a local action diagram with finite orbits, $\mathbf{T}$ be a $\Delta$-tree and $G:=\mathbf{U}_{\mathbf{T}}(\Delta)$. Furthermore, let $\{C_{i}\mid i\in I\}\subseteq C_{\Delta}$ be a collection of minimal translatable circuits representing $C_{\Delta}/\!\sim$. Then
    \begin{displaymath}
        s(G)=\{s(C)\mid C\in C_{\Delta}\}=\{s(C_{i})^{n}\mid i\in I,\ n\in\mathbb{N}\}.
    \end{displaymath}
\end{reptheorem}

With a view towards a GAP package for local action diagrams, we rephrase Bass--Kulkarni's \cite{BK90} unimodularity criterion for groups acting on trees in terms of local action diagrams. In particular, this explicitly identifies unimodularity as a locally determined global property in the sense of Reid--Smith \cite[Section~8]{RS26}. For cocompact $(P)$-closed groups, we reduce said unimodularity criterion to a finite condition and thereby enable computations, see Theorem~\ref{thm:unimodular_fundamental_cycles}.

Finally, we characterise uniscalarity of ($P$)-closed groups in terms of local action diagrams and obtain the following structural description.

\begin{reptheorem}{thm:uniscalar_structure}
    Let $\Delta=(\Gamma,(X_{a}),(G(v)))$ be a local action diagram with finite orbits, $\mathbf{T}=(T,\pi,\mathcal{L})$ be a $\Delta$-tree and $G:=\mathbf{U}_{\mathbf{T}}(\Delta)$. Then $G$ is uniscalar if and only if $G$ is either of type (\emph{Horocyclic}) or contains a compact open normal subgroup.
\end{reptheorem}

This provides one possible answer to a question of Thomas Weigel for classes of t.d.l.c. groups in which uniscalarity implies being compact-by-discrete: the class of non-horocyclic $(P)$-closed groups acting on trees with compact vertex stabilisers.

\subsection*{Acknowledgements}
All authors owe thanks to Colin Reid for the project idea and to Colin Reid and George Willis for many helpful discussions.  Financial support is acknowledged by the first and last authors through the ARC DECRA Fellowship DE210100180, and by the second author through George Willis' ARC Laureate Fellowship FL170100032.
\newpage
\section{Preliminaries}

\subsection{Permutation Groups}

Let $\Omega$ be a set. In this section we collect definitions concerning $\Sym(\Omega)$, the group of bijections from $\Omega$ to itself, called the \textbf{symmetric group} of $\Omega$. When $\Omega=\{1,2,\ldots,n\}$, the group $\Sym(\Omega)$ is also denoted by $S_{n}$.

A subgroup $G\le\Sym(\Omega)$ is a \textbf{permutation group}. For $g\in G$ and $\omega\in\Omega$ we write $g\omega$ for the image of $\omega$ under $g$. The \textbf{orbit} of $\omega$ is $G\omega:=\{g\omega\mid g\in G\}$ and the set of orbits is $G\backslash\Omega:=\{G\omega\mid \omega\in\Omega\}$. For $A\subseteq\Omega$ we define $gA:=\{ga\mid a\in A\}$ and $GA:=\{ga\mid a\in A,\ g\in G\}$. The \textbf{stabiliser} of $\omega \in \Omega$ is $G_{\omega} := \{g \in G \mid g\omega = \omega\}$. The \textbf{pointwise stabiliser} of $A \subseteq \Omega$ is $G_{A} := \{g \in G \mid \forall a\in A:\ ga = a\}$. The \textbf{setwise stabiliser} of $A$ is $G_{\{A\}}:=\{g \in G \mid \forall a\in A: ga\in A\}$. The \textbf{rigid stabiliser} of $A$ is $\mathrm{rist}_{G}(A):=G_{\Omega\backslash A}$. The set $A$ is \textbf{$G$-invariant} if $gA=A$ for all $g\in G$. In this case, the \textbf{restriction} of $G$ to $A$ yields a subgroup of $\Sym(A)$, denoted $G^{A}$, isomorphic to $G/G_{A}$. A group $G\le\Sym(\Omega)$ is \textbf{semiregular} if $G_{\omega}= 1$ for all $\omega \in\Omega$.

\subsection{Permutation Topology}

Let $X$ be a set. The left translates of pointwise stabilisers of finite subsets of $X$ form the basis of the \textbf{permutation topology} on $\Sym(X)$. This topology turns $\Sym(X)$ into a Hausdorff, totally disconnected group, see e.g. \cite{KM08} and \cite{Woe91}. For the reader's convenience, we rephrase the results of the references above in the context relevant here. 

For $n\in\bbN$ and $x=(x_{1},\ldots,x_{n}), y=(y_{1},\ldots,y_{n})\in X^{n}$, define
\begin{displaymath}
 U_{x,y}:=\{g\in \Sym(X)\mid \forall i\in\{1,\ldots,n\}:\ g(x_{i})=y_{i}\}.
\end{displaymath}
Recall that a topological space is \textbf{zero-dimensional} if it admits a basis consisting of closed open sets and that a Hausdorff zero-dimensional space is totally~disconnected.

\begin{lemma}
Let $X$ be a set. Then $\Sym(X)$ is Hausdorff and totally disconnected.
\end{lemma}

\begin{proof}
To see that $\Sym(X)$ is Hausdorff, let $g,h\in\Sym(X)$ be distinct. Then there is an element $x\in X$ such that $g(x)\neq h(x)$. Therefore, $U_{x,g(x)}$ and $U_{x,h(x)}$ are disjoint open sets containing $g$ and $h$ respectively.

For zero-dimensionality, note that the sets $U_{x,y}$ for $x,y\in X^{n}$ and $n\in\bbN$ are open by definition. Now consider $g\in\Sym(X)\backslash U_{x,y}$. Then there is $i\in\{1,\ldots,n\}$ such that $g(x_{i})\neq y_{i}$ and $U_{x,g(x)}\subseteq\Sym(X)\backslash U_{x,y}$ contains $g$. Hence the assertion.
\end{proof}

We now characterise compact sets of the permutation topology on $\Sym(X)$.

\begin{lemma}
Let $X$ be a set equipped with the discrete topology. Then the action map $\Phi:\Sym(X)\times X\to X$ given by $(g,x)\mapsto g(x)$ is continuous.
\end{lemma}

\begin{proof}
Let $Y\subseteq X$ (be open). Then $\Phi^{-1}(Y)=\{(g,x)\in\Sym(X)\times X\mid g(x)\in Y\}$. Hence, if $(g,x)\in\Phi^{-1}(Y)$ then so is the open set $U_{x,g(x)}\times\{x\}$ containing $(g,x)$.
\end{proof}

\begin{proposition}\label{prop:permutation_topology_compact}
Let $X$ be a set and $H\le\Sym(X)$. Then $H$ is compact if and only if $H\le\Sym(X)$ is closed and all its orbits are finite.
\end{proposition}

\begin{proof}
If $H$ is compact, then $H$ is closed in $\Sym(X)$ as $\Sym(X)$ is Hausdorff. Furthermore, $Hx=\Phi|_{H\times\{x\}}$ is compact because $\Phi$ is continuous and hence finite.

Conversely, assume that $H\le\Sym(X)$ is closed and has finite orbits $(X_{i})_{i\in I}$. Then $H\le\prod_{i\in I}\Sym(X_{i})$. Since every $X_{i}$ is finite, $\Sym(X_{i})$ is compact and hence so is $\prod_{i\in I}\Sym(X_{i})$ by Tychonoff's theorem. Therefore, the conclusion follows if we show that the inclusion map $\prod_{i\in I}\Sym(X_{i})\to\Sym(X)$ is continuous. Indeed, an intersection $U_{x,y}\cap\prod_{i\in I}\Sym(X_{i})$ restricts only finitely many factors and hence gives rise to an open subset of the product topology.
\end{proof}

\subsection{Graphs}\label{sec:graphs}

There are inequivalent definitions for many graph-theoretic concepts. The following ones, based on \cite{RS26}, are useful in the context of local action~diagrams.

A \textbf{graph} $\Gamma = (V, A, o, t, r)$ consists of a vertex set $V$, an arc set $A$, \textbf{origin} and \textbf{terminus} maps $o,t : A \to V$ and a \textbf{reversal} map $r : A \to A,\ a\mapsto\overline{a}$ so that $r^{2} = \mathrm{id}$ and $o(r(a)) = t(a)$ for all $a\in A$. For $a\in A$, the pair $\{a, \overline{a}\}$ is an \textbf{edge}. An arc $a \in A$ is a \textbf{loop} if $o(a) = t(a)$. Note that a loop $a\in A$ may be either \textbf{orientable}, i.e. $a\neq\overline{a}$, or \textbf{non-orientable}, i.e. $a=\overline{a}$. A graph is \textbf{simple} if it contains no loops and for every $(u,v) \in V^{2}$ there is at most one $a\in A$ such that $(u,v)=(o(a),t(a))$. In this case, an arc $a \in A$ may be referred to by $(o(a), t(a))$. 

A \textbf{subgraph} of a graph $\Gamma = (V, A, o, t, r)$ is a graph $\Gamma' = (V', A', o', t', r')$ such that $V' \subseteq V$, $A' \subseteq A$, $o = o|_{A'}$, $t = t|_{A'}$, and $r = r|_{A'}$. For a subset $V' \subseteq V$ the subgraph \textbf{induced} by $V'$ has vertex set $V'$ and arc set $\{a \in A\Gamma \mid o(a), t(a) \in V'\}$.

Let $\Gamma$ and $\Gamma'$ be graphs. A \textbf{homomorphism} $\theta$ from $\Gamma$ to $\Gamma'$ is a pair of maps $\theta_{V} : V\Gamma\to V\Gamma'$ and $\theta_{A} : A\Gamma\to A\Gamma'$ such that we have $\theta_{V}(o(a))=o(\theta_{A}(a))$ and $r(\theta_{A}(a))=\theta_{A}(r(a))$ for all $a\in A$. If $\theta_{V}$ and $\theta_{A}$ are bijective then $\theta$ is an \textbf{isomorphism}. An isomorphism from $\Gamma$ to itself is an \textbf{automorphism}. The set of all automorphisms of $\Gamma$ is denoted by $\Aut(\Gamma)$ and forms a group under composition, which we equip with the permutation topology for its action on $V\Gamma$.

For a graph $\Gamma=(V,A,o,t,r)$ and a group $G\le\Aut(\Gamma)$ there is a well-defined \textbf{quotient graph} $G\backslash\Gamma\!:=\!(G\backslash V,G\backslash A,o',t',r')$ with $o'(Ga)\!:=\!Go(a)$, $t'(Ga)\!:=\!Gt(a)$ and $r'(Ga):=Gr(a)$  ($a\in A$), and quotient homomorphism $\pi=\pi_{(\Gamma,G)}:\Gamma\to G\backslash\Gamma$.

For an index set $I \subseteq \mathbb{Z}$, put $I':=\{i \in I \mid i + 1 \in I\}$. A \textbf{path} in $\Gamma$ indexed by $I$ is a sequence of vertices $(v_{i})_{i \in I}$ and edges $(\{a_{i}, \overline{a_{i}}\})_{i \in I'}$ of $\Gamma$ such that $\{a_{i}, \overline{a_{i}}\}$ is an edge between $v_{i}$ and $v_{i+1}$ for all $i \in I'$. The \textbf{length} of a path indexed by $I$ is $|I'|$. A \textbf{directed path} only includes arcs $a_{i}$ such that $o(a_{i}) = v_{i}$ and $t(a_{i}) = v_{i+1}$ for all $i \in I'$. The \textbf{reverse} of a directed path $P$, denoted $\overline{P}$, only contains the arcs $\overline{a_{i}}$ $(i\in I')$. A (directed) path is \textbf{simple} if all its vertices are distinct. A \textbf{ray} (respectively \textbf{line}) in a simple graph is a simple directed path indexed by $\mathbb{N}_{0}$ (respectively $\mathbb{Z}$). Two rays are equivalent if there is another ray that contains infinitely many vertices of both of them. An \textbf{end} of $\Gamma$ is an equivalence class of rays. The set of all ends of $\Gamma$ is denoted by $\partial\Gamma$. A \textbf{cycle} of length $1$ in $\Gamma$ is a vertex together with two distinct, mutually reverse loops. A cycle of length $2$ in $\Gamma$ is a pair of distinct vertices together with two pairs of mutually reverse arcs connecting them. A cycle of length $n\in\mathbb{N}_{\ge 3}$ in $\Gamma$ is a path indexed by $I = \{0, 1, \dots, n\}$ such that $v_{0} = v_{n}$ but otherwise all vertices are distinct. A \textbf{circuit} of length $n\in\mathbb{N}$ is a directed path indexed by $I=\{0,\ldots,n\}$ such that $v_{0}=v_{n}$. We let $C_{\Gamma}$ denote the set of all circuits of $\Gamma$. Given $u,v\in V\Gamma$, the \textbf{distance} $d(u, v)$ between $u$ and $v$ is the minimal length of a path between $u$ and $v$, if one exists, and infinity otherwise. For $v \in V\Gamma$ and $n \geq 1$, the \textbf{ball} $B(v,n)$ of radius $n$ around $v$ is the subgraph induced by $\{u\in V\Gamma\mid d(v,u)\le n\}$. Similarly, the \textbf{sphere} $S(v,n)$ of radius $n$ around $v$ is the subgraph induced by $\{u\in V\Gamma\mid d(v,u)\ =  n\}$. A graph is \textbf{connected} if there is a path between any two distinct vertices. A \textbf{tree} is a non-empty, simple, connected graph without cycles. 

An \textbf{orientation} of a graph $\Gamma$ is a subset $O\subseteq A\Gamma$ such that $A\Gamma=O\sqcup\overline{O}$. A \textbf{partial orientation} of $\Gamma$ is a subset $O \subseteq A\Gamma$ such that $O\cap\overline{O}=\varnothing$. The distance of $v\in V\Gamma$ to a subgraph $\Gamma'$ is $d(v, \Gamma'):=\min\{d(v, v')\mid v'\in V\Gamma'\}$. An arc $a \in A\Gamma$ is \textbf{oriented towards} $\Gamma'$ if $d(o(a), \Gamma') > d(t(a), \Gamma')$. It is oriented towards an end $\xi\in\partial T$ if there is a ray $R \in \xi$ containing $a$.

\subsection{Local Action Diagrams}\label{sec:local_action_diagrams}

Property ($P$) was first introduced as an independence property by Tits \cite{Tit70} to prove the simplicity of certain groups acting on trees. For closed actions, the following easier-to-state generalisation due to Banks--Elder--Willis \cite{BEW15} includes Tits' Property ($P$) as the case $k=1$. 

\begin{definition}[{\cite[Definition 3.1]{BEW15}}]\label{def:pk_closure}
Let $T$ be a tree and $G\leq\Aut(T)$. The $(P_{k})$-\textbf{closure} ($k\in\bbN_{0}$) of $G$ is the group
\begin{displaymath}
	G^{(P_{k})} = \{h \in \Aut(T) \mid \forall v \in V(T)\ \forall X\underset{\!\!\!\text{finite}\!\!\!}{\subseteq} V(B(v,k))\ \exists g \in G:\ g|_{X} = h|_{X}\}.
\end{displaymath}
If $G=G^{(P_{k})}$ then $G$ is $(P_{k})$-\textbf{closed}, or satisfies \textbf{Property} ($P_{k}$).
\end{definition}

Retain the notation of Definition~\ref{def:pk_closure}. When $T$ is locally finite, the quantification over finite sets $X\subseteq V(B(v,k))$ may be replaced by considering just $V(B(v,k))$ itself. In either case, one can show that $\smash{G^{(P_{0})}\ge G^{(P_{1})}\ge G^{(P_{2})}\ge\!\cdots\!\ge G^{(P_{k})}\ge\!\cdots\!\ge\overline{G}\ge G}$ as well as $\smash{\left(G^{(P_{k})}\right)^{(P_{k})}=G^{(P_{k})}}$ and $\smash{\bigcap_{k\in\bbN_{0}}G^{(P_{k})}=\overline{G}}$. This suggests parametrising (closed) subgroups of $\Aut(T)$ by parametrising ($P_{k}$)-closed groups and forming intersections. Finally, we also say \textbf{$(P)$-closed} instead of $(P_{1})$-closed.

In \cite{RS26}, Reid--Smith introduce a powerful parametrisation of ($P$)-closed groups based on the following combinatorial structure. See also the survey article \cite{RS23}.

\begin{definition}[{\cite[Definition 3.1]{RS26}}]\label{def:local_action_diagram}
A \textbf{local action diagram} $\Delta$ is a triple $(\Gamma, (X_{a})_{a \in A\Gamma}, (G(v))_{v \in V\Gamma})$ consisting of
\begin{enumerate}[(i)]
	\item a connected graph $\Gamma$,
	\item pairwise disjoint, non-empty sets $X_{a}$ ($a\in A\Gamma$), and
	\item closed subgroups $G(v)\le\Sym(X_{v})$ ($v\in V\Gamma$), where $\smash{X_{v}:=\bigsqcup_{a\in o^{-1}(v)}X_{a}}$, such that the sets $X_{a}$ ($\smash{a\in o^{-1}(v)}$) are precisely the orbits of $G(v)$.
\end{enumerate}
Each $X_{a}$ is a \textbf{colour set}, its elements are \textbf{colours}, and each $G(v)$ is a \textbf{local action}.
\end{definition}

Reid--Smith introduce a notion of isomorphism for local action diagrams and show in \cite[Theorem 3.3]{RS26} that there is a one-to-one correspondence between isomorphism classes of ($P$)-closed actions ($T,G$), where $T$ is a tree and $G\le\Aut(T)$, and isomorphism classes of local action diagrams. In the following we describe how to pass from a ($P$)-closed action to a local action diagram, and conversely. Passing from a group acting on a tree to a local action diagram is straightforward.

\begin{definition}[{\cite[Definition 3.6]{RS26}}]
Let $T$ be a tree and $G\le\Aut(T)$. Define a local action diagram $\Delta(T,G)=(\Gamma,(X_{a}),(G(v)))$ as follows.
\begin{enumerate}[(i)]
	\item Put $\Gamma:=G\backslash T$ and let $\pi:T\to G\backslash T$ be the quotient map.
	\item For every $v\in V\Gamma$ pick $\tilde{v}\in VT$ with $\pi(\tilde{v})=v$. Given $a\in A\Gamma$ with $o(a)=v$ define $X_{a}:=\{b\in o^{-1}(\tilde{v})\mid\pi(b)=a\}$. Then $X_{v}=o^{-1}(\tilde{v})$.
	\item For every $v\in V\Gamma$ let $G(v)\le\Sym(X_{v})$ be the closure of the permutation group induced on $X_{v}$ by $G_{\tilde{v}}$.
\end{enumerate}
\end{definition}

To pass from a local action diagram to a group acting on a tree we first describe how to obtain a tree projecting onto $\Delta$, and then define a group acting on it.

\begin{definition}[{\cite[Definition 3.4]{RS26}}]
Let $\Delta=(\Gamma,(X_{a}),(G(v)))$ be a local action diagram. A $\Delta$-\textbf{tree} $\mathbf{T}$ is a triple $(T,\pi,\calL)$ consisting of a tree $T$, a surjective graph homomorphism $\pi:T\to\Gamma$ and a map $\calL:AT\to\bigsqcup_{a\in A\Gamma}X_{a}$, termed $\Delta$-\textbf{colouring}, which for all $v\in VT$ and $a\in o^{-1}(\pi(v))$ restricts to bijections $\calL_{v}:o^{-1}(v)\to X_{\pi(v)}$ and $\calL_{v,a}:\{b\in o^{-1}(v)\mid \pi(b)=a\}\to X_{a}$.
\end{definition}

\begin{lemma}[{\cite[Lemma 3.5]{RS26}}]\label{lem:delta_tree_construction}
Let $\Delta$ be a local action diagram. Then there is a $\Delta$-tree $(T,\pi,\calL)$. Given another $\Delta$-tree $(T', \pi', \calL')$ there is a graph isomorphism $\alpha : T \to T'$ such that $\pi' \circ \alpha = \pi$. 
\end{lemma}

We refer to \cite[Lemma 3.5]{RS26} for a full proof of Lemma~\ref{lem:delta_tree_construction} but we include the construction of $\Delta$-trees here for the reader's convenience. See also Example~\ref{ex:delta_tree}.

The following definitions are needed to construct~$T$. Let $\Delta=(\Gamma,(X_{a}),(G(v)))$ be a local action diagram. Given $v\in V\Gamma$ and $c\in X_{v}$ the \textbf{type} $p(c)$ of $c$ is the unique arc $a\in o^{-1}(v)$ with $c\in X_{a}$. A \textbf{coloured path} of length $n\in\mathbb{N}_{0}$ in $\Gamma$ is a sequence $\calC=(c_1, c_2, \dots c_{n})$ of colours such that $o(p(c_{i+1})) = t(p(c_{i}))$ for all $1 \leq i < n$. The \textbf{origin} of $\calC$ is $o(p(c_{1}))$. For all $m\le n$, the path $(c_{1},c_{2},\ldots,c_{m})$ is a \textbf{prefix} of~$\calC$, and any path $(d_{1},d_{2},\ldots,d_{n})$ with $\smash{p(d_{i})=\overline{p(c_{i})}}$ for all $1\le i\le n$ is a \textbf{reverse} of $\calC$.

Pick a base vertex $v_{0}\in V\Gamma$. We inductively define $VT$ as a set of coloured paths in $\Gamma$ originating at $v_{0}$ with chosen reverses. Start with a self-reverse root vertex $()$ and the paths $\{(c)\mid c\in X_{v_{0}}\}$. For every $c\in X_{v_{0}}$ pick $\smash{\overline{c}\in X_{\overline{p(c)}}}$ and put $\smash{\overline{(c)}:=(\overline{c})}$. Now, given $v=(c_{1},c_{2},\ldots,c_{n})$ with reverse $\overline{v}=(d_{1},d_{2},\ldots,d_{n})$ we introduce a new vertex $v_{+c_{n+1}}:=(c_{1},c_{2},\ldots,c_{n},c_{n+1})$ for all $c_{n+1}\in X_{t(p(c_{n}))}\backslash\{d_{n}\}$. Choose any $\smash{d_{n+1}\in X_{\overline{p(c_{n+1})}}}$ and set $\smash{\overline{v_{+c_{n+1}}}}:=(d_{1},d_{2},\ldots,d_{n+1})$.

Depending on the context it is useful to think of vertices of $T$ either as coloured paths or merely as symbols. Define $AT:=AT_{+}\sqcup AT_{-}$, where $AT_{+}$ consists of pairs $(v,w)$ of vertices of $T$ such that $v$ is a prefix of $w$ of length one less than $w$, and $AT_{-}=\{(w,v)\mid (v,w)\in AT_{+}\}$. Given an arc $(v,w)\in AT$, define origin, terminus and reversal by $o((v,w)):=v$, $t((v,w)):=w$ and $r((v,w)):=(w,v)$ respectively. Further, let $\calL(v,w)$ be the last entry of $w$ and $\calL(w,v)$ be the last entry of $\overline{w}$.

We define the graph homomorphism $\pi:T\to\Gamma$ on vertices by $\pi(()):=v_{0}$ and $\pi(v)=t(p(c_{n}))$ for any $v=(c_{1},c_{2},\ldots,c_{n})\in VT$. For arcs, we set $\pi(a):=p(\calL(a))$.

\begin{example}\label{ex:delta_tree}
Consider the local action diagram $\Delta$ on two vertices in Figure~\ref{fig:lad_and_tree}.

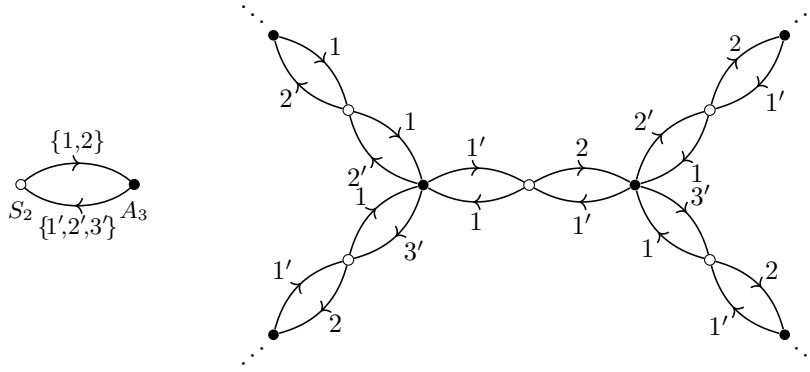
\begin{figure}[ht]
\raisebox{-.5\height}{
\begin{tikzpicture}[scale=1.5]
	\begin{scope}
		[decoration={markings,
			mark=at position 0.52 with {\arrow{>}},
			mark=at position 0.5 with {\node[above]{\small{$\{1,\!2\}$}};}},
		line width=0.6pt]
		\draw [postaction=decorate] (0,0) .. controls (1/4,1/4) and (3/4,1/4) .. (1,0);
	\end{scope}
	\begin{scope}
		[decoration={markings,
			mark=at position 0.52 with {\arrow{>}},
			mark=at position 0.5 with {\node[below]{\small{$\{\!1'\!,\!2'\!,\!3'\!\}$}};}},
		line width=0.6pt]
		\draw [postaction=decorate] (1,0) .. controls (3/4,-1/4) and (1/4,-1/4) .. (0,0);
	\end{scope}
	
	\node (1) at (0,0) {};
	\draw [fill=white] (1) circle [radius=1.25pt];
	\node [below=0.1cm] at (1) {\small{$S_{2}$}};
	
	\node (2) at (1,0) {};
	\draw [fill] (2) circle [radius=1.25pt];
	\node [below=0.1cm] at (2) {\small{$A_{3}$}};
\end{tikzpicture}
}
\hspace{0.5cm}
\raisebox{-.5\height}{
\begin{tikzpicture}[scale=1.4]
	
	\node[fill=white,draw=black,circle,minimum size=4pt,inner sep=0pt] (origin) at (0, 0) {};
	\node[fill,circle,minimum size=4pt,inner sep=0pt] (1) at ($(origin) + (180:1)$) {};
	\node[fill,circle,minimum size=4pt,inner sep=0pt] (2) at ($(origin) + (0:1)$) {}; 
	\node[fill=white,draw=black,circle,minimum size=4pt,inner sep=0pt] (12) at ($(1) + (135:1)$) {};
	\node[fill=white,draw=black,circle,minimum size=4pt,inner sep=0pt] (13) at ($(1) + (225:1)$) {}; 
	\node[fill=white,draw=black,circle,minimum size=4pt,inner sep=0pt] (22) at ($(2) + (45:1)$) {};
	\node[fill=white,draw=black,circle,minimum size=4pt,inner sep=0pt] (23) at ($(2) + (315:1)$) {}; 
	\node[fill,circle,minimum size=4pt,inner sep=0pt] (122) at ($(12) + (135:1)$) {}; 
	\node[fill,circle,minimum size=4pt,inner sep=0pt] (132) at ($(13) + (225:1)$) {}; 
	\node[fill,circle,minimum size=4pt,inner sep=0pt] (222) at ($(22) + (45:1)$) {}; 
	\node[fill,circle,minimum size=4pt,inner sep=0pt] (232) at ($(23) + (315:1)$) {}; 
	\node[circle,minimum size=4pt,inner sep=0pt, anchor=center,rotate=140] (dots1) at ($(122) + (135:0.3) $) {$\dots$}; 
	\node[circle,minimum size=3pt,inner sep=0pt, anchor=center,rotate=220] (dots2) at ($(132) + (225:0.3)$) {$\dots$}; 
	\node[circle,minimum size=3pt,inner sep=0pt, anchor=center,rotate=40] (dots3) at ($(222) + (45:0.3)$) {$\dots$}; 
	\node[circle,minimum size=3pt,inner sep=0pt, anchor=center,rotate=320] (dots4) at ($(232) + (315:0.3)$) {$\dots$}; 
	
	\begin{scope}[decoration={markings, mark=at position 0.53 with {\arrow{>}}},line width=0.6pt]			
		\begin{scope}[label distance=-0.1cm]
			\draw[postaction={decorate},color=black] (origin) to [bend left] node[label=270:$1$] {} (1);
			\draw[postaction={decorate},color=black] (1) to [bend left] node[label=90:$1'$] {} (origin);
			\draw[postaction={decorate},color=black] (origin) to [bend left] node[label=90:$2$] {} (2);
			\draw[postaction={decorate},color=black] (2) to [bend left] node[label=270:$1'$] {} (origin);
		\end{scope}[label distance=-0.25cm]
		\begin{scope}[label distance=-0.25cm]
			\draw[postaction={decorate},color=black] (1) to [bend left] node[label=225:$2'$] {} (12);
			\draw[postaction={decorate},color=black] (1) to [bend left] node[label=315:$3'$] {} (13);	
			\draw[postaction={decorate},color=black] (12) to [bend left] node[label=45:$1$] {} (1);
			\draw[postaction={decorate},color=black] (13) to [bend left] node[label=135:$1$] {} (1);
			\draw[postaction={decorate},color=black] (2) to [bend left] node[label=135:$2'$] {} (22);
			\draw[postaction={decorate},color=black] (2) to [bend left] node[label=45:$3'$] {} (23);
			\draw[postaction={decorate},color=black] (22) to [bend left] node[label=315:$1$] {} (2);
			\draw[postaction={decorate},color=black] (23) to [bend left] node[label=225:$1$] {} (2);
			\draw[postaction={decorate},color=black] (12) to [bend left] node[label=225:$2$] {} (122);
			\draw[postaction={decorate},color=black] (122) to [bend left] node[label=45:$1$] {} (12);
			\draw[postaction={decorate},color=black] (13) to [bend left] node[label=315:$2$] {} (132);
			\draw[postaction={decorate},color=black] (132) to [bend left] node[label=135:$1'$] {} (13);
			\draw[postaction={decorate},color=black] (22) to [bend left] node[above] {\small{$2$}} (222);
			\draw[postaction={decorate},color=black] (222) to [bend left] node[label=315:$1'$] {} (22);
			\draw[postaction={decorate},color=black] (23) to [bend left] node[label=45:$2$] {} (232);
			\draw[postaction={decorate},color=black] (232) to [bend left] node[label=225:$1'$] {} (23);
		\end{scope}
	\end{scope}
\end{tikzpicture}
}
\caption{A local action diagram
\label{fig:lad_and_tree} $\Delta$ and associated $\Delta$-tree.}
\end{figure}

\noindent
Going through the process described above, we see that every $\Delta$-tree is isomorphic to the $(2,3)$-regular tree. Starting at a chosen base vertex, e.g. the central open vertex, note the choices that were made for the reverse labels in Figure~\ref{fig:lad_and_tree}.
\end{example}

Given a $\Delta$-tree $\mathbf{T}=(T,\pi,\calL)$ associated to a local action diagram $\Delta$ we define a ($P$)-closed group acting on it as a subgroup of $\Aut_{\pi}(T):=\{g\in\Aut(T)\mid \pi\circ g=\pi\}$, the group of automorphisms of $T$ that respect $\pi$. To this end, we first define the \textbf{local action} $\sigma_{\mathcal{L}}(g,v)$ of an automorphism $g\in\Aut_{\pi}(T)$ at a vertex $v\in VT$ via
\begin{displaymath}
	\sigma_{\calL,v}:\Aut_{\pi}(T)\to\Sym(X_{\pi(v)}),\ g\mapsto\sigma_{\calL}(g,v):=\calL\circ g\circ\calL|_{o^{-1}(v)}^{-1}.
\end{displaymath}

\begin{definition}[{\cite[Definition 3.8]{RS26}}]\label{def:universal_group}
Let $\Delta=(\Gamma,(X_{a}),(G(v)))$ be a local action diagram and $\mathbf{T}=(T,\pi,\calL)$ be a $\Delta$-tree. The \textbf{universal group of $\mathbf{T}$ with respect to the local actions $(G(v))_{v}$} is
\begin{displaymath}
	\mathbf{U}(\mathbf{T},(G(v))):=\{g\in\Aut_{\pi}(T)\mid \forall v\in VT:\ \sigma_{\calL,v}(g)\in G(\pi(v))\le\Sym(X_{\pi(v)})\}.
\end{displaymath}
\end{definition}

In Example~\ref{ex:delta_tree}, the universal group in the sense of Definition~\ref{def:universal_group} coincides with Smith's \cite{Smi17} group $\mathrm{U}(S_{2},A_{3})$ up to an isomorphism of group actions.

Due to the fact that balls of radius $1$ around distinct vertices of a tree overlap in at most one single edge, elements of $\mathbf{U}(\mathbf{T},(G(v)))$ are readily constructed by extending a given permissible configuration of local actions on a subtree in an arbitrary (compatible) way. For ease of reference, this is made precise in the following lemma, whose proof is inspired by the proof of \cite[Theorem 3.9]{RS26}.

\begin{lemma}\label{lem:universal_extension}
Let $\Delta=(\Gamma,(X_{a}),(G(v)))$ be a local action diagram and $\mathbf{T}=(T,\pi,\calL)$ be a $\Delta$-tree. Further, let $S$ be a subtree of $T$ and let $g:S\to T$ be an injective graph morphism which respects $\pi$. Given permutations $\tau_{v}\in G(\pi(v))$ for every $v\in VS$ such that $\tau_{o(a)}(\mathcal{L}(a))=\mathcal{L}(ga)$ for all arcs $a\in AS$, there is an element $\widetilde{g}\in\mathbf{U}(\mathbf{T},(G(v)))$ such that $\smash{\widetilde{g}|_{S}=g}$ and $\smash{\sigma_{\mathcal{L}}(\widetilde{g},v)=\tau_{v}}$ for every $v\in VS$.
\end{lemma}

\begin{proof}
We construct the automorphism $\smash{\widetilde{g}\in\mathbf{U}(\mathbf{T},(G(v)))}$ inductively on vertices at distance $n\in\mathbb{N}_{0}$ from $S$ along with local permutations $\tau_{v}\in G(\pi(v))$ for the same set of vertices. To begin, set $\smash{\widetilde{g}}|_{S}:=g$ along with the given permutations $(\tau_{v})_{v\in VS}$. Now suppose that $\smash{\widetilde{g}}$ has been defined on vertices at distance up to $n\in\mathbb{N}_{0}$ from $S$ along with local permutations and preserves $\pi$. Given a vertex $v\in VT$ at distance $n+1$ from $S$, let $w\in VT$ denote the unique vertex at distance $n$ from $S$ on the shortest path from $v$ to $S$. Define $\smash{\widetilde{g}(v)}$ as the unique vertex $u$ in $B(gw,1)$ such that $\mathcal{L}((gw,u))=\tau_{w}(\mathcal{L}(w,v))$. Since $\tau_{w}\in G(\pi(w))$ we have that $\pi((w,v))=\pi((gw,u))$ and hence $\pi((v,w))=\pi((u,gw))$ as well as $\pi(v)=\pi(u)$. As a consequence, there is an element $\tau_{v}\in G(\pi(v))$ such that $\tau_{v}(\mathcal{L}((v,w)))=\mathcal{L}((u,gw))$.
\end{proof}

\subsection{Willis theory}\label{sec:willis_theory}

Finally, we recall central definitions and results of Willis theory. Let $G$ be a t.d.l.c.\ group. In~\cite{Wil94}, Willis introduced the notions of \emph{scale} of an automorphism $\alpha$ of $G$ and \emph{tidiness} of a compact open subgroup $U$ of $G$ for $\alpha$.

In this setting, note that $[\alpha(U):\alpha(U)\cap U]\in\bbN$ because $\alpha(U)$ is compact and $\alpha(U)\cap U$ is open in $\alpha(U)$. Following \cite{Wil01}, the \textbf{scale} of $\alpha$ is
\begin{displaymath}\index{scale}\index{$s(\alpha)$|ii}
     s(\alpha):=\min\big\{[\alpha(U):\alpha(U)\cap U]\mid U\le G \text{ compact open}\big\}.
\end{displaymath}
    
A compact open subgroup $U\le G$ is \textbf{minimising} if $[\alpha(U):\alpha(U)\cap U]=s(\alpha)$. It is a cornerstone of Willis theory that a compact open subgroup of $G$ is minimising for $\alpha$ if and only if it is \emph{tidy}. Tidiness is phrased in terms of the following subgroups of $G$. Put $U_{0}:=U$. For $n\in\bbN_{0}$, we define $U_{\pm n}=\bigcap_{k=0}^{n}\alpha^{\pm k}(U)$. Now set
\begin{displaymath}
     U_{+}:=\bigcap_{n\in\bbN_{0}}U_{n}=\bigcap_{k=0}^{\infty}\alpha^{k}(U),\qquad U_{-}:=\bigcap_{n\in\bbN_{0}}U_{-n}=\bigcap_{k=0}^{\infty}\alpha^{-k}(U),
\end{displaymath}
\begin{displaymath}
     \ U_{++}:=\bigcup_{n\in\bbN_{0}}\alpha^{n}(U_{+})\qquad\text{and}\qquad U_{--}:=\bigcup_{n\in\bbN_{0}}\alpha^{-n}(U_{-}).
\end{displaymath}
The following verbal descriptions of the above subgroups may serve as a mnemonic.
\begin{displaymath}
     U_{-}=\left\{\text{\begin{tabular}{c} elements of $U$ whose \\ $\alpha$-trajectory is contained in $U$\end{tabular}}\right\},
\end{displaymath}
\begin{displaymath}
     U_{+}=\left\{\text{\begin{tabular}{c} elements of $U$ whose \\ $\alpha^{-1}$-trajectory is contained in $U$\end{tabular}}\right\},
\end{displaymath}
\begin{displaymath}
     U_{--}=\left\{\text{\begin{tabular}{c} elements of $G$ whose $\alpha$-trajectory \\ is eventually contained in $U$\end{tabular}}\right\},
\end{displaymath}
\begin{displaymath}
     U_{++}=\left\{\text{\begin{tabular}{c} elements of $G$ whose $\alpha^{-1}$-trajectory \\ is eventually contained in $U$\end{tabular}}\right\}.
\end{displaymath}
 The subgroup $U$ is \textbf{tidy above} for $\alpha$ if $U = U_{+}U_{-}$, and \textbf{tidy below} for $\alpha$ if $U_{++}$ is closed. It is \textbf{tidy} for $\alpha$ if it is both tidy above and tidy below for $\alpha$. The cornerstone result of Willis theory now reads as follows.
    
\begin{theorem}[{\cite[Theorem 3.1]{Wil01}}]\label{thm:minimising_tidy}
    Let $G$ be a t.d.l.c.\ group, $\alpha\in\Aut(G)$ and $U\le G$ compact open. Then $U$ is minimising for $\alpha$ if and only if it is tidy for $\alpha$.
\end{theorem}

\section{Axes of Translation}\label{sec:axes_of_translation}

In this section, we show that there are one-to-one correspondences between orbits of translation axes of a $(P)$-closed group as well as conjugacy classes of its translations and certain equivalence classes of \emph{translatable circuits} of the local action diagram. For the purpose of computing the scale, the focus on translations is motivated from the following observation.

\begin{remark}\label{rem:scale_why_translations}
    Let $T$ be a tree and $G\le\Aut(T)$. By \cite[Proposition 3.2]{Tit70}, any automorphism $g\in G$ acts in exactly one of three possible ways: either $g$ fixes a vertex $v\in VT$, or $g$ inverts an arc $a\in AT$, or $g$ translates along an axis of $T$. When $G$ has compact vertex stabilisers, e.g., when $T$ is locally finite, the first two cases imply that $g$ has trivial scale: the groups $G_{v}$ and $G_{\{a,\overline{a}\}}$ respectively are compact, open and normalised by $g$. We therefore focus our attention on translations, which is sufficient when $G$ has compact vertex stabilisers.

    When $G$ does not have compact vertex stabilisers, it may contain automorphisms of non-trivial scale that are not translations. For example, consider the local action diagram $\Delta$ consisting of a single self-reverse loop labelled by the group $\Aut(T_{d})$ ($d\in\mathbb{N}_{\ge 3})$ acting on the set $VT_{d}$, see Figure~\ref{fig:non_locally_finite_example}. Let $\mathbf{T}$ be a $\Delta$-tree and let $\varepsilon\in V\mathbf{T}$ denote the empty coloured path. Note that $S(v,1)\cong VT_{d}$ for all $v\in V\mathbf{T}$.    
    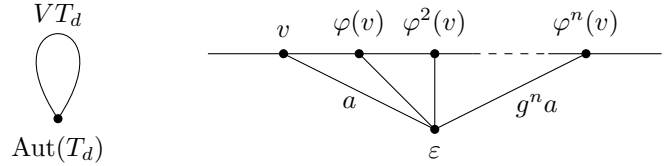
\begin{figure}[ht]
        \raisebox{-.5\height}{
        \begin{tikzpicture}
            \node (0) at (0,0) {};
            \draw [fill] (0) circle [radius=1.5pt];
            \node [below=0.1cm] at (0) {$\Aut(T_{d})$}; 
            \draw (0,0) .. controls (1,1.5) and (-1,1.5) .. (0,0);
            \node (1) at (0,1) {};
            \node [above=0.1cm] at (1) {$VT_{d}$};
        \end{tikzpicture}
        }
        \hspace{0.5cm}
        \raisebox{-.5\height}{
        \begin{tikzpicture}[]
            \node (0) at (0,0) {};
            \draw [fill] (0) circle [radius=1.5pt];
            \node [below=0.1cm] at (0) {$\varepsilon$};
            
            \draw (-3,1) -- (0.5,1);
            \draw [dashed] (0.5,1) -- (1.5,1);
            \draw (1.5,1) -- (3,1);
            \node (v0) at (-2,1) {};
            \draw [fill] (v0) circle [radius=1.5pt];
            \node [above=0.1cm] at (v0) {$v$};
            \node (v1) at (-1,1) {};
            \draw [fill] (v1) circle [radius=1.5pt];
            \node [above=0.1cm] at (v1) {$\varphi(v)$};
            \node (v2) at (0,1) {};
            \draw [fill] (v2) circle [radius=1.5pt];
            \node [above=0.1cm] at (v2) {$\varphi^{2}(v)$};
            \node (vn) at (2,1) {};
            \draw [fill] (vn) circle [radius=1.5pt];
            \node [above=0.1cm] at (vn) {$\varphi^{n}(v)$};

            \draw (0,0) -- (-2,1);
            \node[label={[label distance=-7pt]210:$a$}] at (-1,0.5) {};
            \draw (0,0) -- (-1,1);
            \draw (0,0) -- (0,1);
            \draw (0,0) -- (2,1);
            \node[label={[label distance=-7pt]-60:$g^{n}a$}] at (1,0.5) {};
        \end{tikzpicture}
        }
        \caption{A local action diagram $\Delta$ and a part of $S(\varepsilon,1)$ in $\mathbf{T}$.}
        \label{fig:non_locally_finite_example}
    \end{figure}
    Further, let $\varphi\in\Aut(T_{d})$ be a translation and let $v\in VT_{d}$ lie on the axis of $\varphi$. Put $U':=\Aut(T_{d})_{v}$. Now, consider the arc $a:=(\varepsilon,v)\in A\mathbf{T}$. Then $U:=G_{a}$ has finite vertex orbits on $\mathbf{T}$. In particular, $G$ is locally compact. Next, let $g\in G_{\varepsilon}$ be  an element that acts like $\varphi$ on $S(\varepsilon,1)\cong VT_{d}$. By \cite[Theorem 7.7]{Moe02} and the orbit-stabiliser theorem,
    \begin{align*}
        s_{G}(g)&=\lim_{n\to\infty}[U:U\cap g^{-n}Ug^{n}]^{\frac{1}{n}}=\lim_{n\to\infty}[\Aut(T_{d})_{v}:\Aut(T_{d})_{v,\varphi^{-n}(v)}] \\
        &=\lim_{n\to\infty}[G_{a}:G_{a,g^{-n}a}]=\lim_{n\to\infty}[U':U'\cap \varphi^{-n}U'\varphi^{n}]^{\frac{1}{n}}=s_{\Aut(T_{d})}(\varphi)\neq 1.
    \end{align*}
    To avoid scenarios like the above we restrict ourselves to $(P)$-closed groups with compact vertex stabilisers in Section~\ref{sec:scale_of_translations}.
\end{remark}
	\begin{definition}\label{def:multi_coloured_circuit}
		Let $\Delta = (\Gamma, (X_{a}), (G(v)))$ be a local action diagram and $l \in \mathbb{N}$. A~\textbf{multi-coloured circuit of length $l$} of $\Delta$ is a tuple $C=(a_{i}, S_{i})_{i=0}^{l-1}$ consisting of arcs $a_{i}\in A\Gamma$ and sets $S_{i}\subseteq X_{\overline{a_{i-1}}}\times X_{a_{i}}$ of pairs of colours so that for all $i \in \set*{0,\dots, l\!-\!1}$ we have $t(a_{i}) = o(a_{i+1})$, with indices taken modulo $l$.
	\end{definition}

	Translation axes will be obtained from multi-coloured circuits by picking specific pairs of colours to label the axis. 

	\begin{definition}
        Let $\Delta = (\Gamma, (X_{a}), (G(v)))$ be a local action diagram and $l \in \mathbb{N}$. Further, let $C=(a_{i}, S_{i})_{i=0}^{l-1}$ be a multi-coloured circuit of length $l$ of $\Delta$. A \textbf{cover} of $C$ is a sequence $\smash{\widetilde{C}=(d_{k-1},c_{k})_{k\in\mathbb{Z}}}$ such that $(d_{k-1},c_{k}) \in S_{k\text{ mod }l}$ for all $k \in \mathbb{Z}$.
	\end{definition}

	Covers will correspond to the (pairs of) colours labelling the arcs of an axis. In fact, if $L=(a_{k}, \overline{a_{k}})_{k \in \mathbb{Z}}$ is a line in a $\Delta$-tree $(T,\pi,\mathcal{L})$ then we obtain sequences of colours $c_{k}:=\mathcal{L}(a_{k})$ and $d_{k} \coloneqq \mathcal L(\overline{a_{k}})$, where $k\in\mathbb{Z}$. Since every colour is used exactly once to label the arcs originating at a vertex in a $\Delta$-tree we deduce that $c_{k} \not= d_{k-1}$.  Moreover, if $g$ is a translation of length $l$ along $L$, then for every $k\in\mathbb{Z}$ the pairs $(d_{k-1},c_{k})$ and $(d_{k-1+l},c_{k+l})$ lie in the same orbit for the diagonal action of $G(\pi(o(a_{i})))$ on $X_{\overline{a_{i-1}}}\times X_{a_{i}}$, see Figure~\ref{fig:admissible_col_circuit_aut}. This motivates the following definition.
    
    \begin{figure}[ht]
		\centering%
		\begin{tikzpicture}[scale=1]
			\begin{scope}[xscale=1, yscale=1, thick,decoration={
					markings,
					mark=at position 0.5 with {\arrow{>}}}
				]
				\draw[postaction={decorate}, black] (-2, 0) -- node[below, text height=0.4cm] {$d_{k-1}$} (-3, 0);
				\draw[postaction={decorate}, black] (-2, 0) -- node[below, text height=0.4cm] {$c_{k}$}(-1, 0);
				\draw [fill] (-2, 0) circle [radius=1pt];

				\draw[postaction={decorate}, black] (2, 0) -- node[below, text height=0.4cm] {$d_{k-1+l}$} (1, 0);
				\draw[postaction={decorate}, black] (2, 0) -- node[below, text height=0.4cm] {$c_{k+l}$} (3, 0);
				\draw [fill] (2, 0) circle [radius=1pt];

				\draw[snake arrow] (-0.5, 0) -- node[above, text depth=0.2cm] {$g$} (0.5, 0);
			\end{scope}
		\end{tikzpicture}
		\caption{A requirement on the local actions of translations.}%
		\label{fig:admissible_col_circuit_aut}%
	\end{figure}
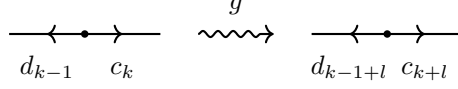
    
    \begin{definition}\label{def:translatable_circuit}
		Let $\Delta = (\Gamma, (X_{a}), (G(v)))$ be a local action diagram and $l \in \mathbb{N}$. Furthermore, let $C=(a_{i}, S_{i})_{i=0}^{l-1}$ be a multi-coloured circuit of length $l$ of $\Delta$. Then ~$C$~is \textbf{translatable} if for all $i\in\{0,\ldots,l-1\}$, the set $S_{i}$ is a non-diagonal orbit of the diagonal action of $G(o(a_{i}))$ on $X_{\overline{a_{i-1}}} \times X_{a_{i}}$.
    \end{definition}

    For the sake of brevity, we will subsequently refer to translatable multi-coloured circuits as \textbf{translatable circuits}.

    \vspace{0.2cm}
    We now show that translatable circuits always give rise to translation axes. Let $\Delta \!=\! (\Gamma, (X_{a}), (G(v)))$ be a local action diagram and $\mathbf T = (T, \pi, \mathcal{L})$ be a $\Delta$-tree. Recall that $VT$ is the set of coloured paths (with chosen reverses) in $\Gamma$ that originate at a chosen base vertex. Given a translatable circuit $C$ of $\Delta$ and a cover $\smash{\widetilde{C}=(d_{k-1},c_{k})_{k\in\mathbb{Z}}}$ of $C$, the set of all start vertices of lines labelled by $\smash{\widetilde{C}}$ in $T$ is
	\[
	S_{\widetilde{C}} \coloneqq \set*{S \in VT \mid \exists \text{ line } (a_{k},\overline{a_{k}})_{k\in\mathbb{Z}} \text{ in $T$}: o(a_{0})=S,\ \mathcal{L}(a_{k})=c_{k},\ \mathcal{L}(\overline{a_{k}})=d_{k} }.
	\]
	Note that $\smash{S_{\widetilde{C}}}$ is non-empty because $C$ is translatable, and that for $\smash{S\in S_{\widetilde{C}}}$ there is a unique line $(a_{k},\overline{a_{k}})_{k\in\mathbb{Z}}$ with the required properties, which we denote by $\smash{L_{S,\widetilde{C}}}$.
	\begin{proposition}\label{prop:translatable_circuit_to_translation}
		Let $\Delta = (\Gamma, (X_{a}), (G(v)))$ be a local action diagram, $\mathbf{T}=(T,\pi,\mathcal{L})$ be a $\Delta$-tree, and $C$ be a translatable circuit of length $l\in\mathbb{N}$ of $\Delta$. Then for any cover $\smash{\widetilde{C}}$ of $C$ and any start vertex $S \in S_{\widetilde{C}}$ the group $\mathbf{U}(\mathbf{T}, (G(v)))$ contains a translation of length $l$ along the line $L_{S,\widetilde{C}}$.
	\end{proposition}

	\begin{proof}
        Let $\smash{L_{S,\widetilde{C}}=(a_{k},\overline{a_{k}})_{k\in\mathbb{Z}}}$ and set $v_{k}:=o(a_{k})$ ($k\in\mathbb{Z}$). Recall that we have $\mathcal{L}(a_{k})=c_{k}$ and $\mathcal{L}(\overline{a_{k}})=d_{k}$ for all $k\in\mathbb{Z}$. Since $C$ is translatable, there is, for every $k\in\mathbb{Z}$, an element $\tau_{k}\in G(\pi(v_{k}))$ such that $\tau_{k}(d_{k-1},c_{k})=(d_{k-1+l},c_{k+l})$. Hence we may define a translation $g\in\mathbf{U}(\textbf{T},(G(v)))$ of length $l$ along $\smash{L_{S,\widetilde{C}}}$ by setting $g(v_{k}):=v_{k+l}$ and $\sigma_{\mathcal{L}}(g,v_{k}):=\tau_{k}$ for all $k\in\mathbb{Z}$, as well as extending $g$ to all of $T$ in an arbitrary (compatible)  way using Lemma~\ref{lem:universal_extension}.
	\end{proof}

    Conversely, translation axes give rise to translatable circuits.        
    \begin{proposition}\label{prop:translation_to_translatable_circuit}
		Let $\Delta = (\Gamma, (X_{a}), (G(v)))$ be a local action diagram, $\mathbf{T}=(T,\pi,\mathcal{L})$ be a $\Delta$-tree, $G:=\mathbf{U}(\mathbf{T},(G(v)))$, and $L$ be a translation axis of $G$ in $T$. Then for any translation $g \in G$ of length $l\in\mathbb{N}$ along $L$ and any $S\in VL$ there is a translatable circuit $C$ of length $l$ of $\Delta$ and a cover $\smash{\widetilde{C}}$ of $C$ such that $S\in \smash{S_{\widetilde{C}}}$ and $\smash{L=L_{S,\widetilde{C}}}$.
	\end{proposition}

	\begin{proof}
        Write $L=(a_{k},\overline{a_{k}})_{k\in\mathbb{Z}}$ with $o(a_{0})=S$, where $k$ increases in the direction of translation of $g$. For all $k\in\mathbb{Z}$, set $v_{k}:=o(a_{k})$ as well as $c_{k}:=\mathcal{L}(a_{k})$ and $d_{k}:=\mathcal{L}(\overline{a_{k}})$. We define a translatable circuit $C$ of $\Delta$ such that $\smash{\widetilde{C}:=(d_{k-1},c_{k})_{k\in\mathbb{Z}}}$ is a cover of $C$ and $\smash{L=L_{S,\widetilde{C}}}$. First, note that for every $i\in\{0,\ldots,l\! -\! 1\}$ the set $\{(d_{i-1+kl},c_{i+kl})\mid k\in\mathbb{Z}\}$ is contained in a non-diagonal orbit $S_{i}$ for the diagonal action of $G(\pi(v_{i}))$ on $X_{\pi(\overline{a_{i-1}})}\times X_{\pi(a_{i})}$, with indices taken modulo $l$, since $g$ is a translation of length $l$ along~$L$. Setting $a_{k}':=\pi(a_{k})$ for all $k\in\{0,\ldots,l\! -\! 1\}$ we see that $C:=(a_{i}',S_{i})_{i=0}^{l-1}$ is the desired translatable circuit.
	\end{proof}

    We now show that Propositions~\ref{prop:translatable_circuit_to_translation} and \ref{prop:translation_to_translatable_circuit} descend to one-to-one correspondences between orbits of translation axes as well as conjugacy classes of translations and certain equivalence classes of translatable circuits.

    There are two fundamental operations on translatable circuits which we use to generate equivalence relations: \emph{shifting} the start point, and \emph{concatenating} a translatable circuit with itself. Formally, let $\Delta=(\Gamma,(X_{a}),(G(v)))$ be a local action diagram and let $C_{\Delta}$ denote the set of translatable circuits of $\Delta$ of any length $l\in\mathbb{N}$. Given $C=(a_{i},S_{i})_{i=0}^{l-1}\in C_{\Delta}$, we declare for any $k\in\mathbb{Z}$ and any $n\in\mathbb{N}$ that
    \begin{displaymath}
        \smash{C\sim_{s} C_{k}:=(a_{i+k},S_{i+k})_{i=0}^{l-1}} \qquad\text{and}\qquad C\preceq_{c}C^{n}:=(a_{i},S_{i})_{i=0}^{nl-1},
    \end{displaymath}
    where indices are taken modulo $l$. We make several observations about these relations:
    \begin{itemize}
        \item the relation $\sim_{s}$ is an equivalence relation on $C_{\Delta}$,
        \item the relation $\preceq_{c}$ is a partial order, and
        \item the symmetric closure of $\preceq_{c}$ is an equivalence relation on $C_{\Delta}$.
    \end{itemize}
    We let $\sim$ denote the equivalence relation on $C_{\Delta}$ generated by $\sim_{s}$ and $\preceq_{c}$. We say that $C\in C_{\Delta}$ is \textbf{minimal} if it is minimal with respect to $\preceq_{c}$. It follows that every equivalence class in $C_{\Delta}/\!\sim$ admits a minimal representative, unique up to $\sim_{s}$.

    \vspace{0.2cm}

    Given a local action diagram $\Delta$ and a $\Delta$-tree $\mathbf{T}$, consider the group $G:=\mathbf{U}_{\mathbf{T}}(\Delta)$. We let $\mathrm{Hyp}(G)$ denote the set of translations in $G$ and let $\mathrm{Axes}(G)$ denote the corresponding set of translation axes. Note that $G$ acts on $\mathrm{Hyp}(G)$ by conjugation and on $\mathrm{Axes}(G)$ through its action on $\mathbf{T}$.   
    \begin{lemma}\label{lem:translations_conjugate}
        Let $T$ be a tree and $G\le\Aut(T)$ be $(P)$-closed. Any two translations in $G$ of the same length and in the same direction along an axis of $T$ are conjugate.
    \end{lemma}

    \begin{proof}
        Let $f,g\in G$ be translations of length $l\in\mathbb{N}$ along an axis $L$ of $T$ with vertices $(v_{k})_{k\in\mathbb{Z}}$. Further, let $\pi:VT\to VL$ denote the closest point projection and let $T_{k}$ denote the subtree of $T$ induced by $\pi^{-1}(v_{k})$. Given $k\in\{0,\ldots,l-1\}$, set $c_{k}:=\id\in\mathrm{rist}_{G}(T_{k})$. Since $G$ is $(P)$-closed we may inductively define elements $c_{k+nl}\in\mathrm{rist}_{G}(T_{k+nl})$ and $c_{k-nl}\in\mathrm{rist}_{G}(T_{k-nl})$ for $n\in\mathbb{N}$ such that
        \begin{displaymath}
            c_{k+nl}|_{T_{k+nl}}=gc_{k+(n-1)l}f^{-1} \quad\text{and}\quad c_{k-nl}|_{T_{k-nl}}=f^{-1}c_{k-(n-1)l}^{-1}g.
        \end{displaymath}
        Again, since $G$ is $(P)$-closed, we obtain an element $c\in G_{L}$ such that $c|_{T_{m}}=c_{m}$ for all $m\in\mathbb{Z}$. Notice that, by definition, $c|_{T_{m+l}}=c_{m+l}=gc_{m}f^{-1}$ for all $m\in\mathbb{Z}$. Thus
        \begin{displaymath}
            cfc^{-1}|_{T_{m}}=c_{m+l}fc_{m}^{-1}=gc_{m}f^{-1}fc_{m}^{-1}=g
        \end{displaymath}
        for all $m\in\mathbb{Z}$, showing that $f$ and $g$ are conjugate in $G$.
    \end{proof}
    \begin{theorem}\label{thm:axes_circuits_correspondence}
		Let $\Delta=(\Gamma,(X_{a}),(G(v)))$ be a local action diagram, $\mathbf T =(T, \pi, \mathcal{L})$ be a $\Delta$-tree and $G:=\mathbf{U}(\mathbf{T},(G(v)))$. Then there are mutually inverse bijections
        \begin{displaymath}
            \xymatrix@C=2.65cm{
            G\backslash\mathrm{Axes}(G) \ar@^{->}@<0.5ex>[r]^-{\text{Proposition~\ref{prop:translation_to_translatable_circuit}}} & C_{\Delta}/\!\sim \ar@^{->}@<0.5ex>[l]^-{\text{Proposition~\ref{prop:translatable_circuit_to_translation}}}
            }
            \hspace{0.15cm}\text{and}\hspace{0.15cm}
            \xymatrix@C=2.65cm{
            G\backslash\mathrm{Hyp}(G) \ar@^{->}@<0.5ex>[r]^-{\text{Proposition~\ref{prop:translation_to_translatable_circuit}}} & C_{\Delta}/\!\sim_{s}, \ar@^{->}@<0.5ex>[l]^-{\text{Proposition~\ref{prop:translatable_circuit_to_translation}}}
            }
        \end{displaymath}
        independent of the choices made in Propositions~\ref{prop:translatable_circuit_to_translation} and \ref{prop:translation_to_translatable_circuit}.
	\end{theorem}

    \begin{proof}
        Let $C=(a_{i},S_{i})_{i=0}^{l-1}\in C_{\Delta}$ be a translatable circuit of $\Delta$ of some length $l\in\mathbb{N}$. We first show that for any two covers $\smash{\widetilde{C}}$, $\smash{\widetilde{C}'}$ of $C$, and start points $\smash{S\in S_{\widetilde{C}}}$, $\smash{S'\in S_{\widetilde{C}'}}$, the two axes $\smash{L_{S,\widetilde{C}}}$, $\smash{L_{S',\widetilde{C}'}}$ of Proposition~\ref{prop:translatable_circuit_to_translation} lie in the same $G$-orbit and their associated translations are conjugate in $G$. Write $\smash{L_{S,\widetilde{C}}=\big(a_{k}},\overline{a_{k}}\big)_{k\in\mathbb{Z}}$ and set
        \begin{displaymath}
            v_{k}:=o\big(a_{k}\big),\quad c_{k}:=\mathcal{L}\big(a_{k}\big),\quad d_{k}:=\mathcal{L}\big(\overline{a_{k}}\big)
        \end{displaymath}
        for all $k\in\mathbb{Z}$, and similarly for $S'$ and $\smash{\widetilde{C}'}$. Since $\smash{\widetilde{C}}$ and $\smash{\widetilde{C}'}$ are both covers of $C$, and $C$ is translatable, there is, for every $k\in\mathbb{Z}$, an element $\smash{\tau_{k}\in G(\pi(v_{k}))=G(\pi(v_{k}'))}$ such that $\smash{\tau_{k}(d_{k-1},c_{k})=(d_{k-1}',c_{k}')}$. We may thus define an element $c\in G$ by setting
        \begin{displaymath}
            c(v_{k}):=v_{k}'\quad\text{and}\quad \sigma_{\mathcal{L}}(c,v_{k}):=\tau_{k},
        \end{displaymath}
        as well as extending $c$ to all of $T$ in an arbitrary compatible way using Lemma~\ref{lem:universal_extension}. Then we have $\smash{cL_{S,\widetilde{C}}=L_{S',\widetilde{C}'}}$ as desired. Moreover, denoting by $g,g'\in\mathrm{Hyp}(G)$ the translations associated to $\smash{L_{S,\widetilde{C}}, L_{S,\widetilde{C}'}}$ by Proposition~\ref{prop:translatable_circuit_to_translation}, we see that $cgc^{-1}$ and $g'$ are conjugate in $G$ by Lemma~\ref{lem:translations_conjugate}. Hence so are $g,g'$.

        \vspace{0.2cm}
        Now, let $C,C'\in C_{\Delta}$ be translatable circuits of $\Delta$ such that $C\sim C'$. Consider the unique minimal translatable circuit $M$ in the equivalence class of $C$ and $C'$ that starts with the same arc as $C$. Let $\smash{\widetilde{M}}=(d_{k-1},c_{k})_{k\in\mathbb{Z}}$ be a cover of $M$. Then $\smash{\widetilde{M}}$ is also a cover of $C$ and there exists an integer $m\in\mathbb{Z}$ such that $\smash{\widetilde{M}_{m}:=(d_{k-1+m},c_{k+m})_{k\in\mathbb{Z}}}$ is a cover of $C'$. Pick a start vertex $\smash{S\in S_{\widetilde{M}}}$ and write $\smash{L_{S,\widetilde{M}}}=(a_{k},\overline{a_{k}})_{k\in\mathbb{Z}}$. Then $S_{m}:=o(a_{m})$ is a start vertex for $\smash{\widetilde{M}_{m}}$ and we see that $\smash{L_{S,\widetilde{C}}}$ and $\smash{L_{S_{m},\widetilde{M}_{m}}}$ coincide as axes of $T$.

        Similarly, suppose $C,C'\in C_{\Delta}$ satisfy $C\sim_{s}C'$ and let $\smash{\widetilde{C}}$ be a cover of $C$. Then there exists an integer $m\in\mathbb{Z}$ such that $\smash{\widetilde{C}_{m}}$ is a cover of $C'$. Pick a start vertex $\smash{S\in S_{\widetilde{C}}}$ and write $\smash{L_{S,\widetilde{C}}=(a_{k},\overline{a_{k}})_{k\in\mathbb{Z}}}$. Then $S_{m}:=o(a_{m})$ is a start vertex for $\smash{\widetilde{C}_{m}}$. As before, $\smash{L_{S,\widetilde{C}}}$ and $\smash{L_{S_{m},\widetilde{C}_{m}}}$ coincide as axes of $T$ and hence the translations associated to these axes by Proposition~\ref{prop:translatable_circuit_to_translation} are conjugate in $G$ by Lemma~\ref{lem:translations_conjugate}.

        \vspace{0.2cm}
        Conversely, consider an axis $L\in\mathrm{Axes}(G)$. We first show that Proposition~\ref{prop:translation_to_translatable_circuit} associates, to any two translations $g,g'\in G$ of lengths $l,l'\in\mathbb{N}$ in the same direction along $L$ and start points $S,S'\in VL$, translatable circuits $C,C'$ such that $C\sim C'$. Write $L=(a_{k},\overline{a_{k}})_{k\in\mathbb{Z}}$ with $o(a_{0})=S$, where $k$ increases in the direction of translation of $g$ and $g'$. Let $m\in\mathbb{Z}$ be such that $S'=o(a_{m})$. For all $k\in\mathbb{Z}$, set $a_{k}':=\pi(a_{k})$. Let $d:=\gcd(l,l')\in\mathbb{N}$. By the Euclidean algorithm, there is a translation of length $d$ along $L$ in the same direction as $g$ and $g'$. Hence we obtain a translatable circuit $D:=(a_{i}',S_{i})_{i=0}^{d-1}$, where $S_{i}$ is the non-diagonal orbit of the diagonal action of $G(\pi(o(a_{i})))$ on $X_{\overline{a_{i-1}}'}\times X_{a_{i}'}$ containing the pairs $\{(\mathcal{L}(\overline{a_{i-1+kl}}),\mathcal{L}(a_{i+kl}))\mid k\in\mathbb{Z}\}$. By construction, $\smash{C=D^{l/d}}$ and $\smash{C'=(C^{l'/d})_{m}}$. In particular, $C\sim D\sim C'$.

        Similarly, given a translation $g\in G$ along an axis $L$ and start points $S,S'\in VL$ the associated translatable circuits $C,C'$ of Proposition~\ref{prop:translation_to_translatable_circuit} satisfy $C\sim_{s}C'$.

        \vspace{0.2cm}
        Now, let $L,L'\in\mathrm{Axes}(G)$ be translation axes of $G$ and suppose $L'=gL$ for some $g\in G$. Let $\varphi\in G$ be a translation of minimal length $l\in\mathbb{N}$ along $L$. Since $g$ conjugates translations along $L$ to translations along $L'$, we conclude that $\varphi':=g\varphi g^{-1}$ is a translation of minimal length along $L'$. Writing $L=(a_{k},\overline{a_{k}})$ with $S:=o(a_{0})$, where $k$ increases in the direction of translation of $\varphi$, put $S':=gS$ and $a_{k}':=ga_{k}$ as well as $\overline{a_{k}}':=g\overline{a_{k}}$ for all $k\in\mathbb{Z}$. Then $L'=(a_{k}',\overline{a_{k}}')_{k\in\mathbb{Z}}$ and $g(\overline{a_{k-1}},a_{k})=(\overline{a_{k-1}}',a_{k}')$ for all $k\in\mathbb{Z}$. Passing to the level of colours, this implies that Proposition~\ref{prop:translation_to_translatable_circuit} associates the same (minimal) translatable circuit to the data $(\varphi,L,S)$ and $(\varphi',L',S')$.

        Similarly, if $\varphi,\varphi'\in\mathrm{Hyp}(G)$ are conjugate translations along axes $L$ and $L'$ of $T$ then they have the same translation length and any conjugator $g\in G$ maps $L$ to $L'$. Defining $S$ and $S'$ as before, we see that Proposition~\ref{prop:translation_to_translatable_circuit} associates the same translatable circuit to $(\varphi,L,S)$ and $(\varphi',L',S')$.

        \vspace{0.2cm}
        Finally, the maps induced by Propositions~\ref{prop:translatable_circuit_to_translation} and \ref{prop:translation_to_translatable_circuit} are mutually inverse in either setting because the choices involved in either proposition become irrelevant at the level of orbits/conjugacy and equivalence classes of circuits by the above.
    \end{proof}

    \section{The Scale of Translations}\label{sec:scale_of_translations}

    We now provide a formula for the scale of a translation in a $(P)$-closed group $G\le\Aut(T)$ with compact vertex stabilisers. Note that by Lemma \ref{prop:permutation_topology_compact} a $(P)$-closed group has compact vertex stabilisers if and only if all its local actions have finite orbits only, an assumption that is satisfied for example when $T$ is locally finite. Given a local action diagram $\Delta=(\Gamma,(X_{a}),(G(v)))$ and a $\Delta$-tree $\mathbf{T}$, the group $G:=\mathbf{U}_{\mathbf{T}}(\Delta)$ thus has compact vertex stabilisers if and only if $X_{a}$ is finite for all $a\in A\Gamma$. We refer to such a diagram as a \textbf{local action diagram with finite orbits}. The formula we develop allows us to describe the set of all scale values of $G$ in terms of its local action diagram, considering that elements fixing a vertex and inversions have trivial scale under the above assumption, see Remark~\ref{rem:scale_why_translations}.
    
    We use the following tidy subgroups to determine the scale of translations.
	\begin{lemma}\label{lem:translation_tidy_subgroup}
        Let $G\leq\Aut(T)$ be ($P$)-closed with compact vertex stabilisers, $g \in G$ be a translation along a line $L$ of $T$, and $v\in VL$. Then $G_{v,gv}\le G$ is tidy for $g$. 
     \end{lemma}

	\begin{proof}
        By assumption, the group $U:=G_{v,gv}\le G$ is compact open. Furthermore,
        \begin{displaymath}
            U_{+}=\bigcap_{k=0}^{\infty}g^{k}Ug^{-k}=G_{v,gv,g^{2}v,\ldots} \quad\text{and}\quad U_{-}=\bigcap_{k=0}^{\infty}g^{-k}Ug^{k}=G_{\ldots,g^{-1}v,v,gv}.
        \end{displaymath}
        Let $u\in U$. Further, let $S$ denote the shortest path from $v$ to $gv$ and let $\pi:VT\to VS$ be the closest point projection. Since $G$ is ($P$)-closed, there is an element $u_{+}\in U_{+}$ which agrees with $u$ on $\pi^{-1}(w)$ for all $w\in VS\backslash\{gv\}$ and fixes $\pi^{-1}(gv)$ pointwise. Similarly, there is an element $u_{-}\in U_{-}$ which agrees with $u$ on $\pi^{-1}(gv)$ and fixes $\pi^{-1}(w)$ pointwise for all $w\in VS\backslash\{gv\}$. By construction, $u=u_{+}u_{-}$ and therefore $U=U_{+}U_{-}$, so $U$ is tidy above for $g$. Now consider
        \begin{displaymath}
            U_{++}=\bigcup_{k=0}^{\infty}g^{k}U_{+}g^{-k}=\bigcup_{k=0}^{\infty}G_{g^{k}v,g^{k+1}(v),\ldots}.
        \end{displaymath}
        That is, $U_{++}$ consists of those elements of $G$ that both fix the end $\xi$ towards which $g$ translates as well as a vertex of $T$. To show that $U_{++}$ is closed, let $h\in G\backslash U_{++}$. Either $h$ does not fix $\xi$, in which case there is an arc $a\in AT$ oriented towards $\xi$ such that $h\cdot a$ is not oriented towards $\xi$, or there is an arc $a\in AT$ which is either inverted or translated by $h$. In any of these cases, the set $hG_{a}$ is an open subset of $G$ contained in $G\backslash U_{++}$.
	\end{proof}

    The following permutation group lemma is used repeatedly in this section.

    \begin{lemma}\label{lem:permutation_group_orbits}
        Let $G\le\Sym(\Omega)$ be a permutation group and let $X,Y\subseteq\Omega$ be orbits of $G$. If $(x,y),(x',y')\in X\times Y$ lie in the same orbit of the diagonal action of $G$ on $X\times Y$ then there is a bijection between $G_{x}\cdot y$ and $G_{x'}\cdot y'$.
    \end{lemma}

    \begin{proof}
        Suppose that $g(x,y)=(gx,gy)=(x',y')$ for some $g\in G$. Then $g$ induces a bijection from $G_{x}\cdot y$ to $G_{x'}\cdot y'$ whose inverse is induced by $g^{-1}\in G$: indeed, given an element $h\in G_{x}$, we have $g(h(y))=ghg^{-1}(gy)=ghg^{-1}(y')\in G_{x'}\cdot y'$.
    \end{proof}

	\begin{proposition}\label{prop:translation_scale}
		Let $\Delta\!=\!(\Gamma,(X_{a}),(G(v)))$ be a local action diagram with finite orbits, $\mathbf{T}\!=\!(T,\pi,\mathcal{L})$ be a $\Delta$-tree, and $G:=\mathbf{U}(\mathbf{T},(G(v)))$. If~$g\in G$ is a translation of length $l\in\mathbb{N}$ along a line $L=(a_{k},\overline{a_{k}})_{k\in\mathbb{Z}}$ in $T$ then
        \begin{displaymath}
            s(g)=\prod_{i=1}^{l}\abs*{G(\pi(o(a_{i})))_{\mathcal{L}(a_{i})} \cdot \mathcal L(\overline{a_{i-1}})}.
        \end{displaymath}
	\end{proposition}

	\begin{proof}
         Let $v_{k}:=o(a_{k})\in VT$ ($k\in\mathbb{Z}$). The subgroup $U:=G_{v_{0},gv_{0}}=G_{v_{0},v_{l}}$ of $G$ is tidy for $g\in G$ by Lemma~\ref{lem:translation_tidy_subgroup}. Hence, by the orbit-stabiliser theorem,
        \begin{displaymath}
            s(g) = \ind{gUg^{-1}}{gUg^{-1}\cap U}=\ind{G_{v_{l},v_{2l}}}{(G_{v_{l},v_{2l}})_{v_{0}}}=\abs*{G_{v_{l},v_{2l}}\cdot v_{0}}.
        \end{displaymath}
        We abbreviate $W:=G_{v_{l},v_{2l}}$. Since $G$ is ($P$)-closed, we deduce using Lemma~\ref{lem:permutation_group_orbits} that
        \begin{displaymath}
            \abs*{W\cdot v_{0}}=\sum_{\substack{v\in W\cdot v_{1} \\ v=hv_{1}}}\abs*{G(\pi(v))_{\mathcal{L}(ha_{1})}\cdot\mathcal{L}(h\overline{a_{0}})}=\abs*{G(\pi(o(a_{1})))_{\mathcal{L}(a_{1})}\cdot\mathcal{L}(\overline{a_{0}})}\cdot\abs*{W\cdot v_{1}},
        \end{displaymath}
        as illustrated by Figure~\ref{fig:scale_translation}. Iterating this argument implies the assertion.
	\end{proof}

    \begin{figure}[ht]
        \begin{tikzpicture}[xscale=1.6,yscale=1.2,
            mid arrow right/.style={
                postaction={decorate},
                decoration={
                    markings,
                    mark=at position 0.5 with {\arrow{>}}
                }
            },
            mid arrow left/.style={
                postaction={decorate},
                decoration={
                    markings,
                    mark=at position 0.5 with {\arrow{<}}
                }
            }
        ]
			\node[fill=black, circle, inner sep=1pt, draw, label=315:$v_{0}$] (v0) at (0,0) {};
			\node[fill=black, circle, inner sep=1pt, draw, label=315:$v_{1}$] (v1) at (1,0) {};
			\node[fill=black, circle, inner sep=1pt, draw, label=315:$v_{2}$] (v2) at (2,0) {};
            \node[fill=black, circle, inner sep=1pt, draw, label=below:$v_{l}$] (vl) at (4,0) {};
            \node[fill=black, circle, inner sep=1pt, draw, label=below:$v_{2l}$] (v2l) at (5,0) {};

            \draw (-0.25,0) -- (v0);
			\draw[mid arrow left] (v0) -- (v1) node[midway, above] {$\overline{a_{0}}$};
			\draw[mid arrow right] (v1) -- (v2) node[midway, above] {$a_{1}$};
            \draw (v2) -- (2.5,0);
			\draw[dashed] (2.5,0) -- (3.5,0);
			\draw (3.5,0) -- (vl) -- (4.25,0);
			\draw[dashed] (4.25,0) -- (4.75,0);
            \draw (4.75,0) -- (v2l) -- (5.25,0);
            \node (L) at (5.5,0) {$L$};

            \path (4,0) ++(163:4) node[fill=black, circle, inner sep=1pt, draw] (hv01) {};
            \path (4,0) ++(155:4) node[fill=black, circle, inner sep=1pt, draw] (hv02) {};
            \path (4,0) ++(147:4) node[fill=black, circle, inner sep=1pt, draw] (hv03) {};
            \path (4,0) ++(155:3) node[fill=black, circle, inner sep=1pt, draw] (v) {};
            \node at ($(v)+(-75:8pt)$) {$v$};
            \path (4,0) ++(155:2) node[fill=black, circle, inner sep=1pt, draw] (hv2) {};
            \node at ($(hv2)+(-70:10pt)$) {$hv_{2}$};
            \path (4,0) ++(160:1.5) node (hv2p5) {};
            \path (4,0) ++(140:0.5) node (hvlm0p5) {};

            \draw[mid arrow left, dashed] (hv01) -- (v);
            \draw[mid arrow left, dashed] (hv02) -- (v);
            \draw[mid arrow left] (hv03) -- (v) node[midway, above right, xshift=-5pt, yshift=1pt] {$h\overline{a_{0}}$};
            \draw[mid arrow right] (v) -- (hv2) node[midway, above right, xshift=-5pt, yshift=1pt] {$ha_{1}$};
            \draw (hv2) -- (hv2p5.center);
            \draw[dashed] (hv2p5.center) -- (hvlm0p5.center);
            \draw (hvlm0p5.center) -- (vl);

            \draw[dotted] (vl) ++(140:4) arc[start angle=140, end angle=220, radius=4];
            \node[right] at ($(vl)+(220:4)$) {$W\cdot v_{0}$};
            \draw[dotted] (vl) ++(140:3) arc[start angle=140, end angle=220, radius=3];
            \node[right] at ($(vl)+(220:3)$) {$W\cdot v_{1}$};
            \draw[dotted] (vl) ++(140:2) arc[start angle=140, end angle=220, radius=2];
			
            \path (4,0) ++(197:4) node[fill=black, circle, inner sep=1pt, draw] (h'v02) {};
            \path (4,0) ++(205:4) node[fill=black, circle, inner sep=1pt, draw] (h'v03) {};
            \path (4,0) ++(200:3) node[fill=black, circle, inner sep=1pt, draw] (v') {};
            \node at ($(v')+(75:8pt)$) {$v'$};
            \path (4,0) ++(205:2) node[fill=black, circle, inner sep=1pt, draw] (h'v2) {};
            \node at ($(h'v2)+(45:8pt)$) {$h'v_{2}$};
            \path (4,0) ++(225:1.5) node (h'v2p5) {};
            \path (4,0) ++(215:0.5) node (h'vlm0p5) {};

            \draw[mid arrow left, dashed] (h'v02) -- (v');
            \draw[mid arrow left] (h'v03) -- (v') node[midway, below right, xshift=-5pt, yshift=-1pt] {$h'\overline{a_{0}}$};
            \draw[mid arrow right] (v') -- (h'v2) node[midway, below, xshift=5pt, yshift=-1pt] {$h'a_{1}$};
            \draw (h'v2) -- (h'v2p5.center);
            \draw[dashed] (h'v2p5.center) -- (h'vlm0p5.center);
            \draw (h'vlm0p5.center) -- (vl);
		\end{tikzpicture}
        \caption{Computing the scale of a translation in a ($P$)-closed group.}
        \label{fig:scale_translation}
    \end{figure}
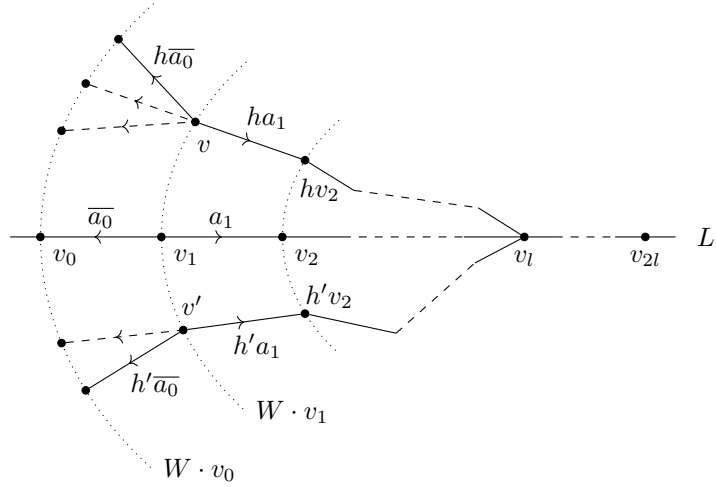

    By Propositions~\ref{prop:translatable_circuit_to_translation} and \ref{prop:translation_to_translatable_circuit}, all translations of ($P$)-closed groups arise from translatable circuits of the underlying local action diagram. We first show that all translations associated to a given circuit by Proposition~\ref{prop:translatable_circuit_to_translation} have the same scale.

    \begin{corollary}\label{cor:scale_translatable_circuit}
        Let $\Delta\!=\!(\Gamma,(X_{a}),(G(v)))$ be a local action diagram with finite orbits, $\mathbf{T}$ be a $\Delta$-tree, and $G:=\mathbf{U}_{\mathbf{T}}(\Delta)$. Further, let $C=(a_{i},S_{i})_{i=0}^{l}$  ($l\in\mathbb{N}$) be a translatable circuit of $\Delta$. 
        Then any two translations of $G$ associated to $C$ via Proposition~\ref{prop:translatable_circuit_to_translation} have the same scale. For any cover $\smash{\widetilde{C}=(d_{k-1},c_{k})_{k\in\mathbb{Z}}}$ of $C$, we have
        \begin{displaymath}
            s(g)=\prod_{i=1}^{l}\abs*{G(o(a_{i}))_{c_{i}}\cdot d_{i-1}}.
        \end{displaymath}
    \end{corollary}

    \begin{proof}
        Consider two covers $\smash{\widetilde{C}}=(d_{k-1},c_{k})_{k\in\mathbb{Z}}$ and $\smash{\widetilde{C}'=(d_{k-1}',c_{k}')_{k\in\mathbb{Z}}}$ of $C$, and let $\smash{S\in S_{\widetilde{C}}}$ as well as $S'\in S_{\widetilde{C}'}$. Let $g$ and $g'$ denote the translations along $\smash{L_{S,\widetilde{C}}}$ and $\smash{L_{S',\widetilde{C}'}}$ associated to this data by Proposition~\ref{prop:translatable_circuit_to_translation}. By Proposition~\ref{prop:translation_scale} and Lemma~\ref{lem:permutation_group_orbits},
        \begin{displaymath}
            s(g)=\prod_{i=1}^{l}\abs*{G(o(a_{i}))_{c_{i}} \cdot d_{i-1}}=\prod_{i=1}^{l}\abs*{G(o(a_{i}))_{c_{i}'} \cdot d_{i-1}'}=s(g'). \qedhere
        \end{displaymath}
    \end{proof}

    Given Corollary~\ref{cor:scale_translatable_circuit}, we write $s(C)$ for the scale of all translations associated to a translatable circuit $C$ of a local action diagram $\Delta$ by Proposition~\ref{prop:translatable_circuit_to_translation}. These scale values behave as follows with respect to the shift and concatenation relation on $C_{\Delta}$.

    \begin{corollary}\label{cor:scale_under_relations}
        Let $\Delta$ be a local action diagram with finite orbits, $\mathbf{T}$ be a $\Delta$-tree and $G:=\mathbf{U}_{\mathbf{T}}(\Delta)$. Further, let $C$ be a translatable circuit of length $l\in\mathbb{N}$ of $\Delta$.
        \begin{enumerate}[(i)]
            \item\label{item:shift} For $k\in\mathbb{Z}$ we have $s(C_{k})=s(C)$.
            \item\label{item:concatenation} For $n\in\mathbb{N}$ we have $s(C^{n})=s(C)^{n}$.
        \end{enumerate}
    \end{corollary}

    \begin{proof}
        For part~\ref{item:shift}, pick translations $g$ and $g'$ representing $C$ and $C_{k}$ respectively. Then use Lemma~\ref{lem:permutation_group_orbits} to match up the factors in the products for $s(g)$ and $s(g')$ given by Corollary~\ref{cor:scale_translatable_circuit}. Similarly, for part~\ref{item:concatenation}, let $g$ and $g'$ be translations representing $C$ and $C^{n}$ respectively. Using Lemma~\ref{lem:permutation_group_orbits}, we see that the $l^{n}$ factors in the product for $s(g')$ given by Corollary~\ref{cor:scale_translatable_circuit} coincide with the factors in the product $s(g)^{n}$.
    \end{proof}

    Corollaries~\ref{cor:scale_translatable_circuit} and \ref{cor:scale_under_relations} allow us to describe the set of all scale values of translations of $\mathrm{U}_{\mathbf{T}}(\Delta)$ purely in terms of minimal translatable circuits of $\Delta$.

    \begin{theorem}\label{thm:scale_values}
         Let $\Delta$ be a local action diagram with finite orbits, $\mathbf{T}$ be a $\Delta$-tree and $G:=\mathbf{U}_{\mathbf{T}}(\Delta)$. Furthermore, let $\{C_{i}\mid i\in I\}\subseteq C_{\Delta}$ be a collection of minimal translatable circuits representing $C_{\Delta}/\!\sim$. Then
         \begin{displaymath}
             s(G)=\{s(C)\mid C\in C_{\Delta}\}=\{s(C_{i})^{n}\mid i\in I,\ n\in\mathbb{N}\}.
         \end{displaymath}
    \end{theorem}

    \begin{proof}
        By Propositions~\ref{prop:translatable_circuit_to_translation} and \ref{prop:translation_to_translatable_circuit}, every translation of $G$ stems from a translatable circuit, hence $s(G)=\{s(C)\mid C\in C_{\Delta}\}$. By Corollary~\ref{cor:scale_under_relations}, the value $s(C)$ is a power of $s(C')$ whenever $C'$ is a minimal representative of $[C]$, hence the assertion.
    \end{proof}

    \begin{example}
        Theorem~\ref{thm:scale_values} recovers the scale of the full automorphism group of a regular tree as first computed in \cite[Section 3]{Wil94}.
        
        \begin{figure}[ht]
            \begin{tikzpicture}
                \node (0) at (0,0) {};
                \draw [fill] (0) circle [radius=1.5pt];
                \node [below=0.1cm] at (0) {$S_d$}; 
                \draw (0,0) .. controls (1,1.5) and (-1,1.5) .. (0,0);
                \node (1) at (0,1) {};
                \node [above=0.1cm] at (1) {$\{1, 2, \dots, d\}$};
            \end{tikzpicture}
            \caption{The local action diagram of $\Aut(T_d).$}
            \label{fig:aut_td_diagram}
        \end{figure}
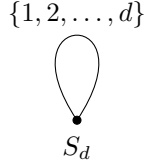

        \noindent
        Consider the translatable circuit $C$ that traverses the loop of Figure~\ref{fig:aut_td_diagram} once and carries the non-diagonal orbit. Then $s(C)=d-1$ by Corollary~\ref{cor:scale_translatable_circuit}. Since $C$ is a minimal representative of the unique equivalence class of translatable circuits, we conclude that $s(\Aut(T_{d}))=\{(d-1)^{n}\mid n\in\mathbb{N}\}$ by Theorem~\ref{thm:scale_values}.
    \end{example}

    \section{Unimodular Groups}
    In \cite[Proposition 3.6]{BK90}, Bass--Kulkarni provide a formula for the modular function and an associated unimodularity criterion for groups acting on locally finite trees without inversion in terms of the associated Bass--Serre graph of groups. See also \cite[Proposition 7.1]{Mar26}. We rephrase this result and its proof in terms of local action diagrams. In doing so, we extend the framework to non-locally finite trees as long as the group has compact vertex stabilisers, treat groups acting with or without inversions via the same argument, and highlight the local nature of the criterion.
    
    In fact, we derive that being unimodular is a locally determined global property for such groups in the sense of Reid--Smith \cite[Section~8]{RS26}. Moreover, with a view towards a work-in-progress GAP package for local action diagrams, we make the unimodularity criterion computable in the cocompact case by reducing it to a finite condition. Finally, we remark that in the case of ($P$)-closed groups the unimodularity criterion can also be obtained using the theory developed in Sections~\ref{sec:axes_of_translation}~and~\ref{sec:scale_of_translations}.

    \vspace{0.2cm}
    Given a locally compact group $G$, let $\Delta_{G}:G\to\mathbb{R}_{>0}^{\ast}$ denote its modular function. That is, for any left Haar measure $\mu$ on $G$, element $g\in G$ and measurable set $E\subseteq G$ we have $\mu(Eg)=\Delta_{G}(g)\mu(E)$.
    
    Next, we consider a local action diagram $\Delta=(\Gamma,(X_{a}),(G(v)))$ with finite orbits. Let $C_{\Gamma}$ denote the set of circuits of $\Gamma$ and let $\mathbb{Z}[A\Gamma]$ be the free abelian group of finitely supported functions from $A\Gamma$ to $\mathbb{Z}$. Define $\nu:C_{\Gamma}\to\mathbb{Z}[A\Gamma]$ by letting $\nu(C)(a)$ be the number of times the arc $a\in A\Gamma$ appears in the circuit $C\in C_{\Gamma}$. Finally, define a group homomorphism $X:\mathbb{Z}[A\Gamma]\to\mathbb{Q}_{>0}^{\ast}$ for $m\in\mathbb{Z}[A\Gamma]$ by setting
    \begin{displaymath}
        X(m):=\prod\nolimits_{a\in A\Gamma}\abs*{X_{a}}^{m(a)}.
    \end{displaymath}

    \begin{proposition}\label{prop:unimodular_circuits}
        Let $T$ be a tree and $G\le\Aut(T)$ have compact vertex stabilisers. Consider $\Delta(T,G)\!=\!(\Gamma,(X_{a}),(G(v)))$ and the projection $\pi:T\to G\backslash T\!=\!\Gamma$. Let $g\!\in\! G$. For $x\in VT$, let $C=(a_{i})_{i=0}^{l-1}$ be a directed path from $x$ to $gx$. Then $\pi C\in C_{\Gamma}$ and
        \begin{displaymath}
            \Delta_{G}(g)=X(\nu(\pi C))/X(\nu(\pi\overline{C})).
        \end{displaymath}        
        Moreover, the following statements are equivalent.
        \begin{enumerate}[(i)]
            \item\label{itm:G_unimodular} The group $G$ is unimodular.
            \item\label{itm:circuit_condition} For all $C\in C_{\Gamma}$ we have $X(\nu(C))=X(\nu(\overline{C}))$.
            \item\label{itm:GP_unimodular} The group $\smash{G^{(P)}}$ is unimodular.
        \end{enumerate}
    \end{proposition}
    \begin{proof}
        Let $\mu$ be a left-invariant Haar measure on $G$. Recall that vertex stabilisers in $G$ are compact open by assumption and therefore have finite non-zero measure. Consequently, given $g\in G$ and $x\in VT$, we have
        \begin{displaymath}
            \Delta_{G}(g)=\frac{\mu(G_{gx}g)}{\mu(G_{gx})}=\frac{\mu(g^{-1}G_{gx}g)}{\mu(G_{gx})}=\frac{\mu(G_{x})}{\mu(G_{gx})}.
        \end{displaymath}
        Furthermore, for any $x\in VT$ and $a\in o^{-1}(x)$ the orbit-stabiliser theorem yields
        \begin{displaymath}
        	\mu(G_{x})=[G_{x}:G_{a}]\mu(G_{a})=|G_{x}\cdot t(a)|\mu(G_{a})=|X_{\pi(a)}|\mu(G_{a}).
        \end{displaymath}
        For $i\in\{0,\ldots,l-1\}$, set $x_{i}:=o(a_{i})$ and $x_{l}:=t(a_{l-1})=gx$. By the above,
        \begin{displaymath}
            \Delta_{G}(g)=\frac{\mu(G_{x_{0}})}{\mu(G_{x_{l}})}=\prod_{i=0}^{l-1}\frac{\mu(G_{x_{i}})}{\mu(G_{x_{l-i}})}=\prod_{i=0}^{l-1}\frac{|X_{\pi(a_{i})}|}{|X_{\pi(\overline{a_{(l-1)-i}})}|}\prod_{i=0}^{l-1}\frac{\mu(G_{a_{i}})}{\mu(G_{\overline{a_{(l-1)-i}}})}=\frac{X(\nu(\pi C))}{X(\nu(\pi\overline{C}))}
        \end{displaymath}
        since $G_{a}=G_{\overline{a}}$ and hence $\mu(G_{a})=\mu(G_{\overline{a}})$ for all $a\in AT$.
        
        Now, suppose that $G$ is unimodular and let $C'\in C_{\Gamma}$. Pick any lift $\smash{C}$ of $C'$ in $T$. Let $x,y\in VT$ denote the start and end vertex of $\smash{C}$ respectively. Then $\pi(x)=\pi(y)$ since $C$ is a circuit and hence there is $g\in G$ such that $gx=y$. By the above,
        \begin{displaymath}
            1=\Delta_{G}(g)=X(\nu(\pi C))/X(\nu(\pi\overline{C}))=X(\nu(C'))/X(\nu(\overline{C'}))
        \end{displaymath}
        as desired. Conversely, if (ii) holds, then $G$ is unimodular by the formula for the modular function. Finally, since $G$ and $G^{(P)}$ have the same local action diagram, the equivalence with (iii) follows.
    \end{proof}

    However, even a finite local action diagram may admit infinitely many (pairwise non-equivalent minimal translatable) circuits. For example, consider a local action diagram $\Delta$ whose underlying graph $\Gamma$ consists of a single vertex with two self-reverse loops $a,b\in A\Gamma$. Given $m,n\in\mathbb{N}$ let $C(m,n)$ denote the circuit of $\Gamma$ given by $n$ times traversing $a$, followed by $m$ times traversing $b$. Then any translatable circuits associated to $\{C(1,n)\mid n\in\mathbb{N}\}$ are minimal and pairwise non-equivalent.

    \vspace{0.2cm}
    We thus establish a reduction of Proposition~\ref{prop:unimodular_circuits} to \emph{fundamental cycles} of $\Gamma$, based on the observation that self-reverse loops and backtracking segments within a circuit contribute equally to the numerator and denominator of $ \Delta_{G}$ in Proposition~\ref{prop:unimodular_circuits}.

    Let $\Delta=(\Gamma,(X_{a}),(G(v)))$ be a local action diagram. We introduce more objects and maps associated to $\Delta$ whose relationship may be summarised as follows.

    \begin{figure}[ht]
        \centering
            \begin{displaymath}
                \xymatrix@C=1.2cm{
                    & C_{\Gamma} \ar@{->}[d]^{\nu} & \\
                    \mathbb{Z}[V\Gamma] \ar@{<-}[d]_{\delta} \ar@{<-}[r]^{\partial} & \mathbb{Z}[A\Gamma] \ar@{->>}[d]^{\mathrm{pr}} \ar[r]^-{X} \ar@{->>}[dl]|{\ \varphi\ } & \mathbb{Q}_{>0}^{\ast} \\
                    \mathbb{Z}[E\Gamma'] \ar@{<<-< }[r]_-{\varphi'} & \mathbb{Z}[A\Gamma]/N
        }
            \end{displaymath}
        \caption{Reduction from arbitrary circuits to fundamental cycles.}
        \label{fig:unimorudlar_reduction}
    \end{figure}
    
    Let $\partial:\mathbb{Z}[A\Gamma]\to\mathbb{Z}[V\Gamma]$ be the group homomorphism defined by $\partial(a):=t(a)-o(a)$ for all $a\in A\Gamma$. We reduce the condition of Proposition~\ref{prop:unimodular_circuits}\ref{itm:circuit_condition} to a cycle basis associated to a spanning tree of the undirected graph $\Gamma'$ that underlies $\Gamma$ after removing self-reverse loops. To this end, let $L:=\{a\in A\Gamma\mid a=\overline{a}\}$ denote the set of self-reverse loops of $\Gamma$. Then we have $E\Gamma':=\{\{a,\overline{a}\}\mid a\in A\Gamma\backslash L\}$. Consider the free abelian group $\mathbb{Z}[E\Gamma']$ and fix an orientation $O\subseteq A\Gamma\backslash L$ of $E\Gamma'$. For every $e\in E\Gamma'$, let $a_{e}\in A\Gamma$ be the unique arc in $O\cap e$. We let $\delta:\mathbb{Z}[E\Gamma']\to\mathbb{Z}[V\Gamma]$ be the classical boundary group homomorphism given by $\delta(e)=t(a_{e})-o(a_{e})$. Finally, let
    \begin{displaymath}
        N:=\langle\{a,b+\overline{b}\mid a,b\in A\Gamma,\ a=\overline{a},\ b\neq\overline{b}\}\rangle
    \end{displaymath}
    be the (normal) subgroup of $\mathbb{Z}[A\Gamma]$ generated by self-reverse loops and backtracking circuits of length $2$. Notice that both $\mathrm{im}(\nu)$ and $N$ are subgroups of $\ker\partial$.

    \begin{lemma}\label{lem:unimodular_reduction}
        Retain the above notation. The surjective group homomorphism
        \begin{displaymath}
            \varphi:\mathbb{Z}[A\Gamma]\to\mathbb{Z}[E\Gamma']\quad \text{defined by}\quad \varphi(a):=\begin{cases}e & a=a_{e} \\ -e & a=\overline{a_{e}} \\ 0 & a=\overline{a}\end{cases}
        \end{displaymath}
        satisfies $\ker\varphi=N$ and therefore induces an isomorphism $\varphi':\mathbb{Z}[A\Gamma]/N\to\mathbb{Z}[E\Gamma']$. Moreover, $\partial=\delta\circ\varphi$ and $\varphi'$ restricts to an isomorphism $\ker\partial/N\to\ker\delta$.
    \end{lemma}

    \begin{proof}
        First note that the definition of $\varphi$ on the basis of $\mathbb{Z}[A\Gamma]$ extends to a unique group homomorphism, which is surjective by definition. We show that $\ker\varphi=N$. By definition, $N\subseteq\ker\varphi$. Conversely, given an element $m\in\ker\varphi\le\mathbb{Z}[A\Gamma]$ we have $0=\varphi(m)(e)=m(a_{e})-m(\overline{a_{e}})$ for all $e\in E\Gamma'$. Therefore,
        \begin{displaymath}
            m=\sum_{a\in L}m(a)a+\sum_{a\in O}m(a)(a+\overline{a})\in N.
        \end{displaymath}       
        Hence there is group isomorphism $\varphi':\mathbb{Z}[A\Gamma]/N\to\mathbb{Z}[E\Gamma']$ such that $\varphi'\circ\mathrm{pr}=\varphi$, where $\mathrm{pr}:\mathbb{Z}[A\Gamma]\to\mathbb{Z}[A\Gamma]/N$ is the quotient homomorphism.

        Concerning the claim $\partial=\delta\circ\varphi$ recall that $A\Gamma=O\sqcup\overline{O}\sqcup L$ and $O=\{a_{e}\mid e\in E\Gamma'\}$. Since $\partial(a)=0$ for all $a\in L$ and $\partial(\overline{a})=-\partial(a)$ for all $a\in A\Gamma\backslash L$ we obtain
        \begin{displaymath}
            \partial(m)\!=\!\!\sum_{a\in A\Gamma}\!m(a)\partial(a)\!=\!\!\sum_{a\in O}(m(a)-m(\overline{a}))\partial(a) \\
            \!=\!\!\!\sum_{e\in E\Gamma'}\!(m(a_{e})-m(\overline{a_{e}}))\delta(e)\!=\!\delta(\varphi(m)).
        \end{displaymath}
        Thus $\partial=\delta\circ\varphi=\delta\circ\varphi'\circ\mathrm{pr}$ and so $\varphi'$ restricts to an isomorphism $\ker\partial/N\to\ker\delta$.
    \end{proof}

    Now, let $T$ be a spanning tree of $\Gamma'$. For every $e\in E\Gamma'\backslash ET$, so-called \emph{chords}, let $P_{e}$ denote the unique simple path in $T$ joining $t(a_{e})$ to $o(a_{e})$ and let
    \begin{displaymath}
        C'_{e}=e+\!\!\sum_{f\in EP_{e}}\varepsilon_{e}(f)f\in\mathbb{Z}[E\Gamma']
    \end{displaymath}
    be the associated \textbf{fundamental cycle}, where $\varepsilon_{e}(f)=1$ if $P_{e}$ traverses $f$ in the direction of $a_{f}$ and $\varepsilon_{e}(f)=-1$ if $P_{e}$ traverses $f$ in the direction of $\overline{a_{f}}$.

    \begin{lemma}\label{lem:cyle_basis}
        Retain the above notation. Then $\{C'_{e}\!\mid\!e\!\in\! E\Gamma'\backslash ET\}$ is a basis of $\ker\delta$.
    \end{lemma}

    \begin{proof}
        Consider a $\mathbb{Z}$-linear combination $\smash{\sum_{e\in E\Gamma'\backslash ET}k_{e}C'_{e}=0}$ for some coefficients $k_{e}\in\mathbb{Z}$. For $e,f\in E\Gamma'\backslash ET$ we have $C'_{e}(f)\neq 0$ if and only if $f=e$. Hence evaluating at $e$ implies $k_{e}=0$ and so the set $\{C'_{e}\!\mid\!e\!\in\! E\Gamma'\backslash ET\}$ is $\mathbb{Z}$-linearly independent. 
        
        To see that $\{C'_{e}\!\mid\!e\!\in\! E\Gamma'\backslash ET\}$ is also spanning, let $C\in\ker\delta$ and consider 
        \begin{displaymath}
            D:=C-\!\!\!\!\sum_{e\in E\Gamma'\backslash ET}\!\!\!\!C(e)C'_{e}\in\ker\delta\le\mathbb{Z}[E\Gamma'].
        \end{displaymath}
        Since $D(e)=0$ for all $e\in E\Gamma'\backslash ET$ by construction, we conclude that $D$ is in fact supported on $ET\cap\mathrm{supp}(D)$, a finite forest. Given that $D\in\ker\delta$ we deduce that $D(f)=0$ for all edges $f\in ET\cap\mathrm{supp}(D)$ incident with a leaf of $T\cap\mathrm{supp}(D)$. Iteratively removing leaves thus shows that $D=0$, as desired.
    \end{proof}    
    
    For every $e\in E\Gamma'\backslash ET$, pick a preimage $C_{e}\in\varphi^{-1}(C'_{e})\cap\mathrm{im}(\nu)\subseteq\mathbb{Z}[A\Gamma]$. Then the group $\mathbb{Z}[A\Gamma]$ is generated by $\{C_{e}\mid e\in E\Gamma'\backslash ET\}\cup N$ by Lemmas~\ref{lem:unimodular_reduction} and \ref{lem:cyle_basis}. Finally, we obtain the following reduction of Proposition~\ref{prop:unimodular_circuits} to fundamental cycles, a finite and hence computable condition in the case of a cocompact action.    
    \begin{theorem}\label{thm:unimodular_fundamental_cycles}
        Retain the notation from above as well as Proposition~\ref{prop:unimodular_circuits}. Then conditions~\ref{itm:G_unimodular}---\ref{itm:GP_unimodular} are equivalent to the following statements.
        \begin{enumerate}[(i)]
            \setcounter{enumi}{3}
            \item\label{itm:fundamental_cycle_condition} For all $e\in E\Gamma'\backslash ET$ we have $X(C_{e})=X(\overline{C_{e}})$.
            \item\label{itm:any_basis_condition} For any collection $(C_{i})_{i\in I}$ of elements $C_{i}\in\mathbb{Z}[A\Gamma]$ that generate $\mathbb{Z}[A\Gamma]$ together with $N$ we have $X(C_{i})=X(\overline{C_{i}})$.
        \end{enumerate}
    \end{theorem}

    \begin{proof}
        Suppose that Proposition~\ref{prop:unimodular_circuits}\ref{itm:circuit_condition} holds. For every $e\in E\Gamma'\backslash ET$ there is, by construction, a circuit $C\in C_{\Gamma}$ such that $\nu(C)=C_{e}$. Hence~\ref{itm:fundamental_cycle_condition}.

        Conversely, suppose~\ref{itm:fundamental_cycle_condition}: $X(C_{e})=X(\overline{C_{e}})$ for all $e\in E\Gamma'\backslash ET$. Consider $C\in C_{\Gamma}$. By the above, there are $k_{e}\in\mathbb{Z}$ and $n\in N$ such that $\smash{\nu(C)=n+\sum_{e\in E\Gamma'\backslash ET}k_{e}C_{e}}$. Since $X$ is a group homomorphism, and $X(n)=X(\overline{n})$ for all $n\in N$, we obtain
        \begin{displaymath}
            X(\nu(C))=X(n)\cdot\!\!\!\!\!\!\prod_{e\in E\Gamma'\backslash ET}\!\!\!\!\!\!X(C_{e})^{k(e)}\!=\!X(\overline{n})\cdot\!\!\!\!\!\!\prod_{e\in E\Gamma'\backslash ET}\!\!\!\!\!\!X(\overline{C_{e}})^{k_{e}}=X(\nu(\overline{C})),
        \end{displaymath}
        which is \ref{itm:circuit_condition}, as desired. Replacing the elements $C_{e}$ with the elements $C_{i}$, the same argument shows that \ref{itm:any_basis_condition} implies \ref{itm:circuit_condition}. Finally, expressing the elements $C_{i}$ in terms of the $C_{e}$ and elements of $N$, the same argument also shows that \ref{itm:fundamental_cycle_condition} implies \ref{itm:any_basis_condition}.
    \end{proof}

    \begin{example}
		Consider the local action diagram $\Delta = (\Gamma, (X_{a}), (G(v)))$ in Figure~\ref{fig:ex_unimodular}. 

		\begin{figure}[ht]
			\begin{tikzpicture}
				\begin{scope}
					[decoration={markings,
						mark=at position 0.52 with {\arrow{>}},
						mark=at position 0.5 with {\node[above]{\small{$\set*{1, 2}$}};}},
					line width=0.6pt]
					\node (0) at (0,0) {};
					\draw [fill] (0) circle [radius=1.5pt];
					\node [left=0.1cm, below=0.15cm] at (0) {\small{$C_2\times S_3$}}; 
					
					\node (1) at (3,0) {};
					\draw [fill] (1) circle [radius=1.5pt];
					\draw [postaction=decorate] (0,0) .. controls (1,0.25) and (2,0.25) .. (3,0);
					\node [below=0.15cm] at (3,0) {\small{$C_{2}\times S_{3}$}};
				\end{scope}

				\begin{scope}
					[decoration={markings,
						mark=at position 0.52 with {\arrow{>}},
					mark=at position 0.3 with {\node[above]{\rotatebox{60}{\small{$\set*{3,4,5}$}}};}},
					line width=0.6pt]
					\node (U) at (1.5, 2.598) {}; 
					\draw [fill] (U) circle [radius=1.5pt];
					\node [above=0.15cm] at (U) {\small{$C_3$}}; 
					\draw [postaction=decorate] (0,0) .. controls (0.283,0.991) and (0.783,1.857) .. (1.5, 2.598); 
				\end{scope}

				\begin{scope}
					[decoration={markings,
						mark=at position 0.52 with {\arrow{<}},
					mark=at position 0.7 with {\node[below=0.2]{\rotatebox{240}{\small{$\{\ast\}$}}};}},
					line width=0.6pt]
					\draw [postaction=decorate] (0,0) .. controls (0.717,0.741) and (1.217,1.607) .. (1.5, 2.598);
				\end{scope}

				\begin{scope}
					[decoration={markings,
						mark=at position 0.52 with {\arrow{<}},
					mark=at position 0.15 with {\node[above=0.6]{\rotatebox{300}{\small{$\set*{a, b}$}}};}},
					line width=0.6pt]
					\draw [postaction=decorate] (3,0) .. controls (2.283,0.741) and (1.783,1.607) .. (1.5, 2.598); 
				\end{scope}

				\begin{scope}
					[decoration={markings,
						mark=at position 0.52 with {\arrow{>}},
					mark=at position 0.825 with {\node[below=0.4]{\rotatebox{-60}{\small{$\{\alpha,\beta,\gamma\}$}}};}},
					line width=0.6pt]
					\draw [postaction=decorate] (3,0) .. controls (2.716,0.991) and (2.217,1.857) .. (1.5, 2.598); 
				\end{scope}

				\begin{scope}
					[decoration={markings,
						mark=at position 0.52 with {\arrow{<}},
						mark=at position 0.5 with {\node[below=0.1]{\small{$\{\ast\}$}};}},
					line width=0.6pt]
					\draw [postaction=decorate] (0,0) .. controls (1,-0.25) and (2,-0.25) .. (3,0);
				\end{scope}

				\begin{scope}
					[decoration={markings,
						mark=at position 0.52 with {\arrow{>}},
						mark=at position 0.5 with {\node[above]{\small{$\set*{c,d,e}$}};}},
					line width=0.6pt]
					\node (R) at (6,0) {};
					\draw [fill] (R) circle [radius=1.5pt];
					\node [left=0.1cm, below=0.15cm] at (R) {\small{$C_2$}}; 
					\draw (6,0) .. controls (7.5,1) and (7.5,-1) .. (6,0);
					\node [right=0.1cm] at (7,0) {\small{$\set*{A, B}$}};
					
					\draw [postaction=decorate] (3,0) .. controls (4,0.25) and (5,0.25) .. (6,0);
				\end{scope}
				\begin{scope}
					[decoration={markings,
						mark=at position 0.52 with {\arrow{<}},
						mark=at position 0.5 with {\node[below=0.1]{\small{$\{\ast\}$}};}},
					line width=0.6pt]
					\draw [postaction=decorate] (3,0) .. controls (4,-0.25) and (5,-0.25) .. (6,0);
				\end{scope}
			\end{tikzpicture}
            \vspace{-0.3cm}
			\caption{A local action diagram of a unimodular (P)-closed group.}
			\label{fig:ex_unimodular}
		\end{figure}
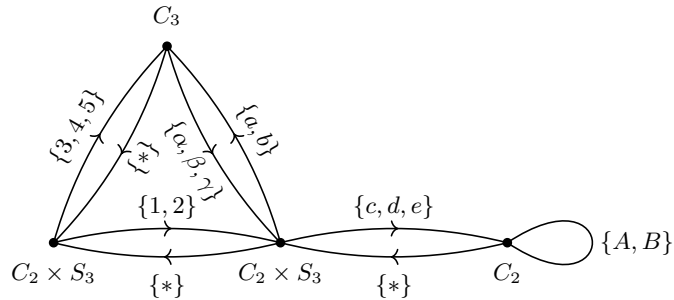

		Note that $\mathbf{U}(\Delta)$ contains edge inversions due to the self-reverse loop on the rightmost vertex. After removing this self-reverse loop, we obtain a spanning tree and see that there is only one associated fundamental cycle, $C$, see Figure~\ref{fig:unimodular-st}.

		\begin{figure}[ht]
        \raisebox{-.5\height}{
			\begin{tikzpicture}
				\begin{scope}
					[decoration={markings,
						mark=at position 0.52 with {\arrow{>}},
						mark=at position 0.5 with {\node[above]{};}},
					line width=0.6pt]
					\node (0) at (0,0) {};
					\draw [fill] (0) circle [radius=1.5pt];
					\node [left=0.1cm, below=0.15cm] at (0) {}; 
					
					\node (1) at (3,0) {};
					\draw [fill] (1) circle [radius=1.5pt];
					\draw [postaction=decorate] (0,0) .. controls (1,0.25) and (2,0.25) .. (3,0);
					\node [below=0.15cm] at (3,0) {};
				\end{scope}

				\begin{scope}
					[decoration={markings,
						mark=at position 0.52 with {\arrow{>}},
					mark=at position 0.3 with {\node[above]{};}},
					line width=0.6pt]
					\node (U) at (1.5, 2.598) {}; 
					\draw [fill] (U) circle [radius=1.5pt];
					\node [above=0.15cm] at (U) {}; 
					\draw [postaction=decorate] (0,0) .. controls (0.283,0.991) and (0.783,1.857) .. (1.5, 2.598); 
				\end{scope}

				\begin{scope}
					[decoration={markings,
						mark=at position 0.52 with {\arrow{<}},
					mark=at position 0.7 with {\node[below=0.2]{};}},
					line width=0.6pt]
					\draw [postaction=decorate] (0,0) .. controls (0.717,0.741) and (1.217,1.607) .. (1.5, 2.598);
				\end{scope}

				\begin{scope}
					[decoration={markings,
						mark=at position 0.52 with {\arrow{<}},
						mark=at position 0.5 with {\node[below=0.1]{};}},
					line width=0.6pt]
					\draw [postaction=decorate] (0,0) .. controls (1,-0.25) and (2,-0.25) .. (3,0);
				\end{scope}

				\begin{scope}
					[decoration={markings,
						mark=at position 0.52 with {\arrow{>}},
						mark=at position 0.5 with {\node[above]{};}},
					line width=0.6pt]
					\node (R) at (6,0) {};
					\draw [fill] (R) circle [radius=1.5pt];
					\node [left=0.1cm, below=0.15cm] at (R) {}; 
					\draw [postaction=decorate] (3,0) .. controls (4,0.25) and (5,0.25) .. (6,0);
				\end{scope}
				\begin{scope}
					[decoration={markings,
						mark=at position 0.52 with {\arrow{<}},
						mark=at position 0.5 with {\node[below=0.1]{};}},
					line width=0.6pt]
					\draw [postaction=decorate] (3,0) .. controls (4,-0.25) and (5,-0.25) .. (6,0);
				\end{scope}
		\end{tikzpicture}}
        \hspace{0.5cm}
        \raisebox{-.535\height}{
			\begin{tikzpicture}
				\begin{scope}
					[decoration={markings,
						mark=at position 0.52 with {\arrow{>}},
						mark=at position 0.5 with {\node[above]{\small{$\set*{1, 2}$}};}},
					line width=0.6pt]
					\node (0) at (0,0) {};
					\draw [fill] (0) circle [radius=1.5pt];
					\node [left=0.1cm, below=0.15cm] at (0) {}; 
					
					\node (1) at (3,0) {};
					\draw [fill] (1) circle [radius=1.5pt];
					\draw [postaction=decorate] (0,0) .. controls (1,0.25) and (2,0.25) .. (3,0);
					\node [below=0.15cm] at (3,0) {};
				\end{scope}

				\begin{scope}
					[decoration={markings,
						mark=at position 0.52 with {\arrow{>}},
					mark=at position 0.3 with {\node[above]{\rotatebox{60}{\small{$\set*{3, 4, 5}$}}};}},
					line width=0.6pt]
					\node (U) at (1.5, 2.598) {}; 
					\draw [fill] (U) circle [radius=1.5pt];
					\node [above=0.15cm] at (U) {}; 
					\draw [postaction=decorate] (0,0) .. controls (0.283,0.991) and (0.783,1.857) .. (1.5, 2.598); 
				\end{scope}

				\begin{scope}
					[decoration={markings,
						mark=at position 0.52 with {\arrow{<}},
					mark=at position 0.7 with {\node[below=0.2]{\rotatebox{240}{\small{$\{\ast\}$}}};}},
					line width=0.6pt]
					\draw [postaction=decorate] (0,0) .. controls (0.717,0.741) and (1.217,1.607) .. (1.5, 2.598);
				\end{scope}

				\begin{scope}
					[decoration={markings,
						mark=at position 0.52 with {\arrow{<}},
					mark=at position 0.15 with {\node[above=0.6]{\rotatebox{300}{\small{$\set*{a,b}$}}};}},
					line width=0.6pt]
					\draw [postaction=decorate] (3,0) .. controls (2.283,0.741) and (1.783,1.607) .. (1.5, 2.598); 
				\end{scope}

				\begin{scope}
					[decoration={markings,
						mark=at position 0.52 with {\arrow{>}},
					mark=at position 0.825 with {\node[below=0.4]{\rotatebox{-60}{\small{$\{\alpha, \beta, \gamma\}$}}};}},
					line width=0.6pt]
					\draw [postaction=decorate] (3,0) .. controls (2.716,0.991) and (2.217,1.857) .. (1.5, 2.598); 
				\end{scope}

				\begin{scope}
					[decoration={markings,
						mark=at position 0.52 with {\arrow{<}},
						mark=at position 0.5 with {\node[below=0.1]{\small{$\{\ast\}$}};}},
					line width=0.6pt]
					\draw [postaction=decorate] (0,0) .. controls (1,-0.25) and (2,-0.25) .. (3,0);
				\end{scope}

			\end{tikzpicture}}
            \vspace{-0.3cm}
			\caption{A spanning tree of $\Gamma$ and the associated fundamental cycle.}
			\label{fig:unimodular-st}
		\end{figure}
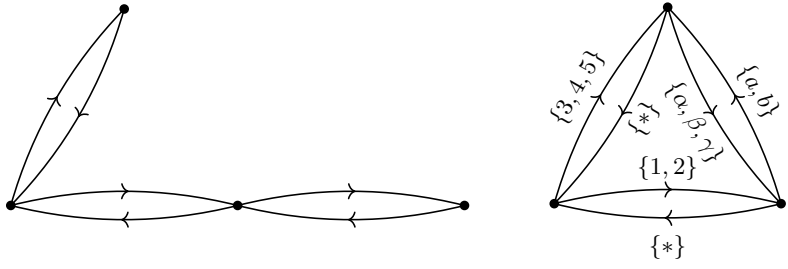

		Note that $X(C) = X(\overline{C}) = 6$ and hence $\mathbf{U}(\Delta)$ is unimodular by Theorem~\ref{thm:unimodular_fundamental_cycles}.
    \end{example}

    Finally, we show how in the case of ($P$)-closed groups, the formula for the modular function of Proposition~\ref{prop:unimodular_circuits} and a version of the criterion~\ref{itm:circuit_condition} may also be obtained using the theory developed in Sections~\ref{sec:axes_of_translation} and \ref{sec:scale_of_translations}. This is based on the fact that, for a totally disconnected locally compact group $G$, the modular function is given by $\Delta_{G}(g)=s(g^{-1})/s(g)$ for all $g\in G$, see \cite[Corollary 1]{Wil94}.

    To this end, let $G\le\Aut(T)$ be a ($P$)-closed group with compact vertex stabilisers. For the purpose of computing the modular function, it suffices to consider translations of $G$ by Remark~\ref{rem:scale_why_translations}. When $g\in G$ translates along a line $L$ of $T$, then $g^{-1}$ is a translation of the same length along $L$ in the opposite direction. In terms of translatable circuits, the inverse of a translation associated to some circuit is associated to the \emph{reverse} circuit, as defined below.

    \begin{definition}
        Let $\Delta=(\Gamma,(X_{a}),(G(v)))$ be a local action diagram and let $C=(a_{i},S_{i})_{i=0}^{l-1}$ be a multi-coloured circuit of length $l\in\mathbb{N}$ of $\Delta$. The \textbf{reverse} of $C$ is the multi-coloured circuit $\overline{C}:=(\overline{a_{l-1-i}},\overline{S_{l-1-i}})_{i=0}^{l-1}$, where $\overline{S_{k}}=\{(c,d)\mid (d,c)\in S_{k}\}$ for all $k\in\{0,\ldots,l\!-\!1\}$.
    \end{definition}

    Given a local action diagram $\Delta=(\Gamma,(X_{a}),(G(v)))$ as well as a translatable circuit $C=(a_{i},S_{i})_{i=0}^{l-1}$ of length $l\in\mathbb{N}$ of $\Delta$ and a cover $\smash{\widetilde{C}=(d_{k-1},c_{k})_{k\in\mathbb{Z}}}$ of $C$, the tuple $(c_{l-k},d_{l-1-k})_{k\in\mathbb{Z}}$ is a cover of $\overline{C}$. Put $v_{i}:=o(a_{i})$ for all $i\in\{0,\ldots,l\!-\!1\}$. When $\Delta$ has finite orbits, Corollary~\ref{cor:scale_translatable_circuit} implies that comparing $s(g)$ and $s(g^{-1})$ for a translation $g$ associated to $C$ via Proposition~\ref{prop:translatable_circuit_to_translation} amounts to comparing
    \begin{align*}
        s(g)\!&=\!\prod\limits_{i=0}^{l-1}\abs*{G(o(a_{i}))_{c_{i}}\cdot d_{i-1}}=\prod_{i=0}^{l-1}\abs*{G(v_{i})_{c_{i}}\cdot d_{i-1}}, \\
        s(g^{-1})\!&=\!\prod\limits_{i=0}^{l-1}\abs*{G(o(\overline{a_{l-1-i}}))_{d_{l-1-i}}\cdot c_{l-i}}\!=\!\prod_{i=0}^{l-1}\abs*{G(v_{l-i})_{d_{l-1-i}}\cdot c_{l-i}}\!=\!\prod\limits_{i=0}^{l-1}\abs*{G(v_{i})_{d_{i-1}}\cdot c_{i}}.
    \end{align*}

    The following lemma reduces this analysis to comparing sizes of orbits along $C$.
    \begin{lemma}\label{lem:unimodular_orbit_sizes}
        Let $\Delta=(\Gamma,(X_{a}),(G(v)))$ be a local action diagram with finite orbits. Further, let $C=(a_{i})_{i=0}^{l-1}$ be a circuit of length $l\in\mathbb{N}$ of $\Delta$ and $(d_{i-1},c_{i})\in X_{\overline{a_{i-1}}}\times X_{a_{i}}$ for all $i\in\{0,\ldots,l-1\}$. Then
        \begin{displaymath}
            \frac{\prod_{i=0}^{l-1}\abs*{G(o(a_{i}))_{d_{i-1}}\cdot c_{i}}}{\prod_{i=0}^{l-1}\abs*{G(o(a_{i}))_{c_{i}}\cdot d_{i-1}}} = \frac{\prod_{i=0}^{l-1}\abs*{X_{a_{i}}}}{\prod_{i=0}^{l-1}\abs*{X_{\overline{a_{i}}}}}.
        \end{displaymath}
    \end{lemma}

    \begin{proof}
        Put $v_{i}:=o(a_{i})$ for all $i\in\{0,\ldots,l\!-\!1\}$. By the orbit-stabiliser theorem, all of the above factors can be rewritten as indices: for all $i\in\{0,\ldots,l\!-\!1\}$ we have $\abs*{G(o(a_{i}))_{c_{i}}\cdot d_{i-1}}\!=\![G(v_{i})_{c_{i}}:G(v_{i})_{c_{i},d_{i-1}}]$ and $\abs*{X_{a_{i}}}\!=\!\abs*{G(v_{i})\cdot c_{i}}=[G(v_{i}):G(v_{i})_{c_{i}}]$, and similarly for the denominators. Hence the assertion follows from writing
        \begin{align*}            [G(v_{i}):G(v_{i})_{c_{i},d_{i-1}}]&=[G(v_{i}):G(v_{i})_{c_{i}}]\cdot [G(v_{i})_{c_{i}}:G(v_{i})_{c_{i},d_{i-1}}], \\
        [G(v_{i}):G(v_{i})_{c_{i},d_{i-1}}]&=[G(v_{i}):G(v_{i})_{d_{i-1}}]\cdot [G(v_{i})_{d_{i-1}}:G(v_{i})_{c_{i},d_{i-1}}],
        \end{align*}
        and taking the product over all $i\in\{0,\ldots,l\!-\!1\}$.
    \end{proof}

    By virtue of Lemma~\ref{lem:unimodular_orbit_sizes}, we obtain the following specialisation of Proposition~\ref{prop:unimodular_circuits}.
    \begin{proposition}\label{prop:unimodular_translatable_circuits}
        Let $\Delta=(\Gamma,(X_{a}),(G(v)))$ be a local action diagram with finite orbits, $\mathbf{T}$ be a $\Delta$-tree and $G:=\mathbf{U}_{\mathbf{T}}(\Delta)$. Further, let $g\in G$ be a translation with associated (minimal) translatable circuit $(a_{i},S_{i})_{i=0}^{l-1}$ via Proposition~\ref{prop:translation_to_translatable_circuit}. Then
        \begin{displaymath}
            \Delta_{G}(g)=\frac{\prod_{i=0}^{l-1}\abs*{X_{a_{i}}}}{\prod_{i=0}^{l-1}\abs*{X_{\overline{a_{i}}}}}
        \end{displaymath}
        In particular, $G$ is unimodular if and only if $\prod_{i=0}^{l-1}\abs*{X_{a_{i}}}=\prod_{i=0}^{l-1}\abs*{X_{\overline{a_{i}}}}$ for all (minimal) translatable circuits of $\Delta$.
    \end{proposition}

    \begin{proof}
        The formula for $\Delta_{G}(g)$ follows from Lemma~\ref{lem:unimodular_orbit_sizes} and the preceding discussion. The criterion then follows from Remark~\ref{rem:scale_why_translations}, Proposition~\ref{prop:translatable_circuit_to_translation} and Corollary~\ref{cor:scale_under_relations}.
    \end{proof}

    \section{Uniscalar (\texorpdfstring{$P$}{P})-closed Groups}

    In this section, we characterise uniscalarity of $(P)$-closed groups in terms of local action diagrams. By drawing structural conclusions from uniscalarity, we provide one possible answer to a question of Thomas Weigel for classes of t.d.l.c. groups in which uniscalarity implies being compact-by-discrete. Recall that, by \cite[Theorem 2.5]{RS26}, groups acting on trees fall into six mutually exclusive types; see also \cite[Section~2]{CT24}\todo{update reference when known}. The following lemma reduces the question of uniscalarity to the general type.

    \begin{lemma}\label{lem:type_uniscalar}
        \hspace{-1.8mm} Let $T$ be a tree. If $G\!\le\!\Aut(T)$ has compact vertex stabilisers and~type
        \begin{enumerate}[(i)]
            \item (\emph{Fixed Vertex}), (\emph{Inversion}), (\emph{Lineal}) or (\emph{Horocyclic}), then $G$ is uniscalar.
            \item (\emph{Focal}) then $G$ is not uniscalar.
        \end{enumerate}
    \end{lemma}

    \begin{proof}
         If $G\le\Aut(T)$ is of type (\emph{Fixed Vertex}), (\emph{Inversion}) or (\emph{Horocyclic}) then it does not contain any translations. Hence $G$ is uniscalar by Remark~\ref{rem:scale_why_translations}.
         
         If $G$ is of type (\emph{Lineal}), the stabiliser of an arc belonging to the line between the two fixed ends fixes the entire line. It is thus a compact open subgroup normalised by any translation in $G$ and hence $G$ is uniscalar in this case as well.

         Finally, suppose $G$ is of type (\emph{Focal}). Let $\omega_{+}\in\partial T$ denote the unique end fixed by $G$ and let $g\in G$ be a translation of length $l\in\mathbb{N}$ along a line $L$ of $T$ towards $\omega_{+}$. Furthermore, let $\omega_{-}$ denote the opposite end of $L$. Since $G$ fixes a unique end, there is an element $h'\in G$ such that $h'\omega_{-}\neq\omega_{-}$. Using $h'$, we show that there exists an element $h\in G$ which fixes a vertex $v\in VL$ and moves the arc $b\in o^{-1}(v)$ oriented towards $\omega_{-}$. Letting $a:=gb\in AL$ we then have $h\in G_{a}$ and thus $|G_{a}\cdot b|\ge 2$.
         
         Since $h'\in G$ fixes $\omega_{+}$, it cannot be an inversion. If $h'$ fixes a vertex then, since $h'$ fixes $\omega_{+}$, it also fixes a vertex $v\in VL$. Since $h'\omega_{-}\neq\omega_{-}$ we conclude that $h'$ moves the arc $b\in o^{-1}(v)$ oriented towards $\omega_{-}$. Thus we may use $h:=h'$.

         Otherwise, $h'$ is a translation of some length $l'$ along an axis $L'$ which eventually coincides with $L$. Passing to $(h')^{-1}$ if necessary, we may assume that $h'$ translates towards $\omega_{+}$. Let $v\in VL\cap VL'$ denote the vertex at which $L$ and $L'$ branch. Then $\smash{h:=(h')^{-l}g^{l'}\in G}$ fixes $v$ and moves the arc $b\in o^{-1}(v)$ oriented towards $\omega_{-}$.

         Now, given $n\in\mathbb{N}_{0}$, define $h_{n}:=g^{-n}hg^{n}\in G$ as well as $v_{n}:=g^{-n}v\in VL$ and $b_{n}:=g^{-n}b\in o^{-1}(v_{n})$. Then we have $h_{0}=h$ as well as $v_{0}=v$ and $b_{0}=b$. Moreover, $h_{n}$ fixes $v_{n}$ and moves $b_{n}$. Finally, set $U_{n}:=G_{g^{-n}a}$ for all $n\in\mathbb{N}_{0}$. Notice that $h_{n}\in U_{n}$ and $U_{n}\le U_{n-1}\le\ U_{n-2}\le\cdots\le U_{1}\le U_{0}$ for all $n\in\mathbb{N}_{0}$. Furthermore, $[U_{n}:U_{n+1}]\ge 2$ for all $n\in\mathbb{N}_{0}$ by the above. Using \cite[Theorem 7.7]{Moe02}, we obtain
         \begin{displaymath}
             s(g)=\!\lim_{n\to\infty}[U_{0}:U_{0}\cap g^{-n}U_{0}g^{n}]^{\frac{1}{n}}=\!\lim_{n\to\infty}[U_{0}:U_{n}]^{\frac{1}{n}}=\!\lim_{n\to\infty}\left(\prod_{k=0}^{n}[U_{k}:U_{k+1}]\right)^{\frac{1}{n}}\!\!\!\ge 2,
         \end{displaymath}
         showing that $G$ is not uniscalar.
    \end{proof}
        
    In the case of ($P$)-closed groups, the action type can be determined from the local action diagram, see \cite[Theorem 2.6]{CT24}\todo{update reference when known}, and the following statement characterises uniscalarity for the general type. It is instructive to compare this to the characterisation of discreteness for $(P)$-closed groups, see \cite[Theorem 3.1]{CT24}\todo{update reference when known}.
    \begin{proposition}\label{prop:general_type_uniscalar}
        Let $\Delta=(\Gamma,(X_{a}),(G(v)))$ be a local action diagram with finite orbits, $\mathbf{T}$ be a $\Delta$-tree and $G:=\mathbf{U}_{\mathbf{T}}(\Delta)$. Suppose that $G$ is of type (\emph{General}) and let $\Gamma'$ be the unique smallest cotree of $\Gamma$. Then $G$ is uniscalar if and only if for all $v\in V\Gamma'$ the restriction of $G(v)\!\le\!\Sym(X_{v})$ to $\bigsqcup_{a\in o^{-1}(v)\cap A\Gamma'}X_{a}\!\subseteq\! X_{v}$ is semiregular.
    \end{proposition}

    \begin{proof}
        First, suppose that for every $v\in V\Gamma'$ the restriction of $G(v)\le\Sym(X_{v})$ to $\smash{\bigsqcup_{a\in o^{-1}(v)\cap A\Gamma'}X_{a}\subseteq X_{v}}$ is semiregular. Since $\Delta$ has finite orbits, it suffices to consider the scale of translations of $G$, which in turn stem from translatable circuits by Section~\ref{sec:scale_of_translations}. Let $C=(a_{i},S_{i})_{i=0}^{l-1}$ be a translatable circuit of length $l$ of $\Delta$. We argue that $a_{i}\in A\Gamma'$ for all $i\in\{0,\ldots,l-1\}$. Otherwise, by virtue of $\Gamma'$ being a cotree, there would exist a pair of (cyclically) consecutive arcs $a_{i-1},a_{i}$ ($i\in\{0,\ldots,l-1\}$ such that $a_{i}=\overline{a_{i-1}}$ is oriented towards $\Gamma'$ and hence $S_{i}\subseteq X_{\overline{a_{i-1}}}\times X_{a_{i}}$ is necessarily the diagonal orbit of $G(o(a_{i}))$, contradicting that $C$ is translatable. Combining Corollary~\ref{cor:scale_translatable_circuit} with the semiregularity assumption, we conclude that any translation $g\in G$ associated to $C$ via Proposition~\ref{prop:translatable_circuit_to_translation} satisfies $s(g)=1$.
        
        Conversely, assume that there is a vertex $v\in V\Gamma'$ such that $G(v)$ does not act semiregularly when restricted to $\bigsqcup_{a\in o^{-1}(v)\cap A\Gamma'}X_{a}$. Then there are colours $c,d\in X_{v}$ such that $p(c),p(d)\in A\Gamma'$ and $|G(v)_{c}\cdot d|\ge 2$. In particular, $|X_{p(d)}|\ge 2$. Using \cite[Lemma 3.3]{CT24}\todo{update reference when known}, construct a translatable circuit $(a_{i},S_{i})_{i=0}^{l-1}$ such that $o(a_{0})=v$ and $(d,c)\in S_{0}$. Then the scale of any translation $g\in G$ associated to $C$ via Proposition~\ref{prop:translatable_circuit_to_translation} satisfies $s(g)\ge 2$ by Corollary~\ref{cor:scale_translatable_circuit}.
    \end{proof}
    Next, we examine the structure of uniscalar $(P)$-closed groups. In all cases but the (\emph{Horocyclic}) one, uniscalarity is exhibited through a compact open normal subgroup.
    \begin{theorem}\label{thm:uniscalar_structure}
         Let $\Delta=(\Gamma,(X_{a}),(G(v)))$ be a local action diagram with finite orbits, $\mathbf{T}=(T,\pi,\mathcal{L})$ be a $\Delta$-tree and $G:=\mathbf{U}_{\mathbf{T}}(\Delta)$. Then $G$ is uniscalar if and only if $G$ is either of type (\emph{Horocyclic}) or contains a compact open normal subgroup.
    \end{theorem}

    \begin{proof}
        Groups of type (\emph{Horocyclic}) are uniscalar by Lemma~\ref{lem:type_uniscalar}. In general, the existence of a compact open normal subgroup guarantees uniscalarity by Section~\ref{sec:willis_theory}.
        
        If $G$ is of type (\emph{Fixed Vertex}) or (\emph{Inversion}) then $G$ is compact by Lemma~\ref{prop:permutation_topology_compact}. If $G$ is of type (\emph{Lineal}), let $L$ denote the line between the two fixed ends of $G$ and consider $K:=G_{L}$. For any arc $a\in AL$ we have $K=G_{a}$ and hence $K$ is a compact open subgroup of $G$. Moreover, since $L$ is $G$-invariant, $K$ is normal in $G$. 

        By Lemma~\ref{lem:type_uniscalar}, groups of type (\emph{Focal}) are never uniscalar.

        Finally, suppose $G$ is of type (\emph{General}). Then $\Delta$ contains a unique minimal cotree $\Gamma'$ by \cite[Theorem 2.6]{CT24}\todo{update reference when known}. Let $T':=\pi^{-1}(\Gamma')$ denote its preimage in $T$, the unique smallest $G$-invariant subtree of $T$. Consider the subgroup $K:=G_{T'}$ of $G$. Given any arc $a\in AT'$, Proposition~\ref{prop:general_type_uniscalar} implies that $K=G_{a}$ is compact open. Since $T'$ is $G$-invariant, $K$ is also normal in $G$. 
    \end{proof}
    
    Complementing Theorem~\ref{thm:uniscalar_structure}, we note that horocyclic $(P)$-closed groups with compact vertex stabilisers never contain a compact open normal subgroup.

    \begin{proposition}
         Let $\Delta=(\Gamma,(X_{a}),(G(v)))$ be a local action diagram with finite orbits, $\mathbf{T}=(T,\pi,\mathcal{L})$ be a $\Delta$-tree and $G:=\mathbf{U}_{\mathbf{T}}(\Delta)$. If $G$ is of type (\emph{Horocyclic}) then $G$ does not contain a compact open normal subgroup.
    \end{proposition}

    \begin{proof}
        Suppose that $K\unlhd G$ is compact and normal. Since $K$ is compact and acts without inversion, it fixes a vertex $v\in VT$. Arguing as in \cite[Section~3.5]{CT24}\todo{update reference when known}, the $G$-orbit of $v$, which is contained in the horosphere of $v$, must be infinite. Since $K$ is normal in $G$, it fixes said orbit. But, repeating the argument from \cite[Section~3.5]{CT24}\todo{update reference when known}, for any finite set $S\subseteq VT$ the basic open subgroup $H_{S}$ acts non-trivially on said orbit due to how $T$ is constructed from~$\Delta$. Hence $K$ is not open.
    \end{proof}

        For the following, retain the notation of Theorem~\ref{thm:uniscalar_structure} and its proof. We comment on the structure of the discrete quotient $G/K$ of $G$ by the constructed compact open normal subgroup $K$ and the associated group extension $\smash{\xymatrix@C=0.5cm{\! K \ar@{ >->}[r] & G \ar@{->>}[r] & G/K.}}$

        First, we discuss the structure of the discrete group $G/K$. In the (\emph{Fixed Vertex}) and (\emph{Inversion}) cases, the discrete quotient is trivial, whereas in the (\emph{Lineal}) case we have $G/K=G/G_{L}\cong G^{L}\cong\mathbb{Z}$. In the (\emph{General}) case, by \cite[Lemma~2.21]{RS26}, we have $\smash{G/K=G/G_{T'}\cong G^{T'}=\mathbf{U}_{\mathbf{T'}}(\Delta')}$  where $\smash{\Delta'=(\Gamma',(X_{a}),(G(v)^{X_{v}'}))}$ is the local action diagram that arises from $\Gamma'$, that is, $\smash{X_{v}'=\bigsqcup_{a\in o^{-1}(v)\cap A\Gamma'} X_{a}}$, and $\mathbf{T'}=(T',\pi,\mathcal{L})$ is the associated $\Delta'$-tree. Since the permutation group $\smash{G(v)^{X_{v}'}}$ is semiregular for all $v\in V\Gamma'$ by Proposition~\ref{prop:general_type_uniscalar}, we have $\smash{\mathbf{U}_{\mathbf{T'}}(\Delta')_{\tilde{v}}\cong G(\pi(\tilde{v}))^{X_{\pi(\tilde{v})}'}}$ for all $\tilde{v}\in VT'$, and $\mathbf{U}_{\mathbf{T'}}(\Delta')_{\tilde{a}}=\{\id\}$ for all $\tilde{a}\in AT'$. Applying the Bass--Serre structure theorem \cite[\S 5.4 Theorem 13]{Ser80}, see also \cite[Theorem 2.33]{RS26}, to the inversion-free subdivision of $\Gamma'$ defined in \cite[Section 2.5]{RS26} now gives the following proposition, as the stabiliser in $\mathbf{U}_{\mathbf{T'}}(\Delta')$ of any subdivision vertex is isomorphic to $\mathbb{Z}/2\mathbb{Z}$ by semiregularity of the local actions. 
        
        \begin{proposition}
            Let $\Delta=(\Gamma,(X_{a}),(G(v)))$ be a local action diagram with finite orbits and $\mathbf{T}$ be a $\Delta$-tree. If $\mathbf{U}_{\mathbf{T}}(\Delta)$ is discrete and of (\emph{General}) type then
            \begin{displaymath}
                \mathbf{U}_{\mathbf{T}}(\Delta) \cong \bigast_{v \in V\Gamma}G(v) \ast \bigast_{a \in L}\mathbb{Z}/2\mathbb{Z} \ast F_r,
            \end{displaymath}
            where $L$ is the set of all self-reverse loops of $\Gamma$ and $F_r$ is the free group of rank $r$, where $r \in \mathbb{N}_{0} \cup \{\infty\}$ is the topological rank of $\Gamma$ after removing $L$. In particular, if $\mathbf{T}$ is locally finite and $\mathbf{U}_{\mathbf{T}}(\Delta)$ is cocompact then $\mathbf{U}_{\mathbf{T}}(\Delta)$ is virtually free.
        \end{proposition}

        Free products of finitely many finite groups and a free group of finite rank are called \emph{plain} and form an active area of research in combinatorial and geometric group theory, see for example \cite{EP22}. In the context of uniscalar $(P)$-closed groups, retaining the notation introduced earlier, we see that
        \begin{displaymath}
                G/K \cong \mathbf{U}_{\mathbf{T'}}(\Delta') \cong \bigast_{v \in V\Gamma'}G(v)^{X_{v}'} \ast \bigast_{a\in L}\mathbb{Z}/2\mathbb{Z} \ast F_r, 
        \end{displaymath}
        where $L$ is the set of all self-reverse loops of the cotree $\Gamma'$ and $F_r$ is the free group of rank $r$, where $r \in \mathbb{N}_{0} \cup \{\infty\}$ is the topological rank of the cotree $\Gamma'$ after removing~$L$.

        Next, we comment on whether the extension $\smash{\xymatrix@C=0.5cm{\! K \ar@{ >->}[r] & G \ar@{->>}[r] & G/K}}$ splits. In the (\emph{Lineal}) case, $G/K\cong\mathbb{Z}$ and hence the extension splits. The following example shows that it need not split in the (\emph{General}) case.
		
		\begin{example}
		    Let $d\in\mathbb{N}_{\ge 3}$. Consider the group $(\mathbb{Z}/d^{2}\mathbb{Z},+)$, along with the natural projection $\mathbb{Z}/d^{2}\mathbb{Z}\twoheadrightarrow\mathbb{Z}/d\mathbb{Z}$, and the action of $\mathbb{Z}/d^{2}\mathbb{Z}$ on $\mathbb{Z}/d^{2}\mathbb{Z}\sqcup\mathbb{Z}/d\mathbb{Z}$ given by addition and addition after projection respectively. Then the local action diagram of Figure~\ref{fig:non_split_general_type} gives rise to a uniscalar $(P)$-closed group $G$ of (\emph{General}) type.

		\begin{figure}[ht]
			\begin{tikzpicture}
				\begin{scope}
					[decoration={markings,
						mark=at position 0.52 with {\arrow{>}},
						mark=at position 0.5 with {\node[above]{\small{$\mathbb{Z}/d^{2}\mathbb{Z}$}};}},
					line width=0.6pt]
					\node (0) at (0,0) {};
					\draw [fill] (0) circle [radius=1.5pt];
					\node [below=0.15cm] at (0) {$\mathbb{Z}/d^{2}\mathbb{Z}$}; 
					\draw (0,0) .. controls (-1.5,1) and (-1.5,-1) .. (0,0);
					\node [left=0.1cm] at (-1,0) {\small{$\mathbb{Z}/d\mathbb{Z}$}};
					
					\node (1) at (2,0) {};
					\draw [fill] (1) circle [radius=1.5pt];
					\draw [postaction=decorate] (0,0) .. controls (0.75,0.25) and (1.25,0.25) .. (2,0);
					\node [right=0.1] at (2,0) {$\{\id\}$};
				\end{scope}
				\begin{scope}
					[decoration={markings,
						mark=at position 0.52 with {\arrow{<}},
						mark=at position 0.5 with {\node[below=0.1]{\small{$\{1\}$}};}},
					line width=0.6pt]
					\draw [postaction=decorate] (0,0) .. controls (0.75,-0.25) and (1.25,-0.25) .. (2,0);
				\end{scope}
			\end{tikzpicture}
            \vspace{-0.3cm}
			\caption{A uniscalar $(P)$-closed group of (\emph{General}) type.}
			\label{fig:non_split_general_type}
		\end{figure}
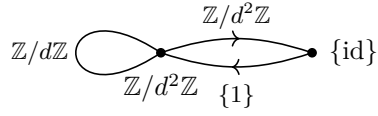
		
		\noindent
		The associated $\Delta$-tree $\mathbf{T}=(T,\pi,\mathcal{L})$ consists of a $G$-invariant $d$-regular subtree $T'$ with a bouquet of $d^{2}$ edges attached to each vertex, and $K:=G_{T'}\cong\prod_{v\in VT}\mathbb{Z}/d\mathbb{Z}$, shifting each edge bouquet individually by a multiple of $d$ steps. The restricted action $\smash{G^{T'}}\cong G/G_{T'}=G/K$ is the universal group $\mathbf{U}_{\mathbf{T}'}(\mathbb{Z}/d\mathbb{Z})$ for the natural action of $\mathbb{Z}/d\mathbb{Z}$ on itself by addition. Now suppose that $\psi:\mathbf{U}_{\mathbf{T'}}(\mathbb{Z}/d\mathbb{Z})\cong G/K\to G$ is an injective homomorphism. We will show that $\mathrm{pr}\circ\psi\neq\id$.
		
		Let $v\in VT'$ and consider $H:=\mathbf{U}_{T'}(\mathbb{Z}/d\mathbb{Z})_{v}\lesssim G/K$. Then $\mathbb{Z}/d\mathbb{Z}\cong H\cong \psi(H)$. In particular, $\psi(H)\le G$ is finite and thus either fixes a vertex or preserves an edge.
		
		Suppose $\psi(H)\le G_{u}$ for some $u\in VT$. If $u\in VT'$ then since $|\psi(h)|$ divides $d$ for all $h\in H$, so does $|\sigma_{\mathcal{L}}(\psi(h),u)|$ and hence $\sigma_{\mathcal{L}}(\psi(H),u)\in d\mathbb{Z}/d^{2}\mathbb{Z}$. By semiregularity, $\psi(H)$ fixes $T'$ and thus $\mathrm{pr}\circ\psi(H)=\{\id\}$. In particular, $\mathrm{pr}\circ\psi\neq\id$. If $u\in VT\backslash VT'$ then $\psi(H)$ also fixes the unique vertex $w\in VT'$ with $(u,w)\in AT$.
		
		Now suppose $\psi(H)\le G_{\{a,\overline{a}\}}$ for some arc $a\in AT$. By the above, we may assume that $o(a),t(a)\in VT'$. Put $H_{2}:=\{h^{2}\mid h\in H\}\le H$. Then $\psi(H_{2})\le G_{a}$ fixes $T'$ by semiregularity. Hence $\mathrm{pr}\circ\psi(H_{2})=\{\id\}$. But $H_{2}\cong 2\mathbb{Z}/d\mathbb{Z}\lesssim H\cong\mathbb{Z}/d\mathbb{Z}$ is non-trivial when $d\ge 3$. We conclude again that $\mathrm{pr}\circ\psi\neq\id$ and therefore we see that the extension $\smash{\xymatrix@C=0.5cm{\! K \ar@{ >->}[r] & G \ar@{->>}[r] & G/K}}$ does not split. 
		\end{example}

    At the 2023 workshop ``Totally disconnected locally compact groups: local to global'' held at Universit{\"a}t M{\"u}nster, Thomas Weigel asked for classes of t.d.l.c. groups in which uniscalarity implies being compact-by-discrete, see \cite{NPR+23}. Theorem~\ref{thm:uniscalar_structure} provides one possible answer: the class of non-horocyclic $(P)$-closed groups acting on trees with compact vertex stabilisers. It is unclear to the authors whether the attribute of being $(P)$-closed may be dropped from these assumptions.

    \begin{question}
        Is there a uniscalar group acting on a tree with compact vertex stabilisers of type (\emph{General}) that does not contain a compact open normal subgroup?
    \end{question}

    Finally, we remark that being uniscalar is not a locally determined global property in the sense of Reid--Smith \cite[Section 8]{RS26}. For example, given a permutation group $F\le S_{d}$, the group $\mathrm{U}_{2}(\Gamma(F))\le\Aut(T_{d})$ introduced in \cite[Section 4.4.1]{Tor23} is discrete, and hence uniscalar, whereas its $(P)$-closure, the Burger--Mozes group $\mathrm{U}_{2}(\Gamma(F))^{(P)}=\mathrm{U}_{1}(F)$ is uniscalar if and only if $F$ is semiregular, see \cite[Proposition~6.20]{GGT18} or Proposition~\ref{prop:general_type_uniscalar}. However, a group whose $(P)$-closure is uniscalar must itself be uniscalar.

    \begin{proposition}
        Let $T$ be a tree and $G\le\Aut(T)$ have compact vertex stabilisers. If $\smash{G^{(P)}}$ is uniscalar then so is $G$, and $G$ is either of type (\emph{Horocyclic}) or contains a compact open normal subgroup.
    \end{proposition}

    \begin{proof}
        Since $G$ and $G^{(P)}$ have the same type by \cite[Proposition 2.4]{CT24}\todo{update reference when known}, the first statement follows from Lemma~\ref{lem:type_uniscalar} except for type (\emph{General}). The second statement is immediate for types (\emph{Fixed Vertex}), (\emph{Inversion}) and (\emph{Focal}). In the (\emph{Lineal)} case, the stabiliser of the line connecting the two fixed ends is compact open normal.

        Finally, suppose that $G$ and $\smash{G^{(P)}}$ are of type (\emph{General}) and $G^{(P)}$ is uniscalar. Let $T'$ be the unique smallest subtree of $T$ that is $\smash{G^{(P)}}$-invariant. Then $\smash{G^{(P)}_{T'}}$ is a compact open normal subgroup of $G^{(P)}$ by the proof of Theorem~\ref{thm:uniscalar_structure}. Since $G\le G^{(P)}$ and the subspace topology on $G$ coincides with the permutation topology, the subgroup $\smash{G_{T'}=G_{T'}^{(P)}\cap G}$ is a compact open normal subgroup of $G$. Hence $G$ is uniscalar.
    \end{proof}

	\bibliographystyle{amsalpha}
	\bibliography{scale}

\end{document}